\documentclass[11pt]{article}
\usepackage{amsmath,amsthm,amsfonts,amssymb,bm,wasysym}
\usepackage{epsfig}
\usepackage[usenames]{color}
\usepackage{verbatim}
\usepackage{hyperref}
\usepackage{multicol}
\usepackage{comment}
\usepackage{float}
\usepackage{graphicx}
\usepackage[utf8]{inputenc}
\def\bs{\backslash}

\parskip \medskipamount
\newtheorem{theorem}{Theorem}[section]

\numberwithin{equation}{section}
\newtheorem{lemma}[theorem]{Lemma}
\newtheorem{proposition}[theorem]{Proposition}
\newtheorem{remark}[theorem]{Remark}

\numberwithin{equation}{section}

\def\N{\mathbb{N}}
\def\Z{\mathbb{Z}}

\def\R{\mathbb{R}}
\def\C{\mathcal{C}}

\def\cR{\mathcal{R}}
\def\cO{\mathcal{O}}
\def\V{\mathcal{V}}
\def\CC{\mathcal{C}}

\def\bE{\mathbb{E}}
\def\bP{\mathbb{P}}

\def\D{\mathcal{D}}

\renewcommand{\phi}{\varphi}
\renewcommand{\epsilon}{\varepsilon}

\def\K{\mathcal{K}}

\def\RR{\mathcal{R}}
\def\C{{\mathcal C}}

\allowdisplaybreaks

\newcommand{\1}{{\text{\Large $\mathfrak 1$}}}

\newcommand{\var}{\operatorname{var}}

\def\reff#1{(\ref{#1})}

\newcommand{\tn}{|\kern-.1em|\kern-0.1em|}

\newcommand{\cp}{\mathrm{Cap}}

\newcommand{\cc}[1]{\mathrm{Cap}\left(#1\right)}
\newcommand{\ccc}[1]{\overline{\mathrm{Cap}}\left(#1\right)}

\def\bP{\mathbb{P}}

\begin{document}
\title{\bf Deviations for the Capacity of the Range of a Random Walk}

\author{Amine Asselah \thanks{
LAMA, Univ Paris Est Creteil, Univ Gustave Eiffel, UPEM, CNRS, F-94010, Cr\'eteil, France; amine.asselah@u-pec.fr } \and
Bruno Schapira\thanks{Aix-Marseille Universit\'e, CNRS, Centrale Marseille, I2M, UMR 7373, 13453 Marseille, France;  bruno.schapira@univ-amu.fr} 
}
\date{}
\maketitle
\begin{abstract}
We obtain estimates for large and moderate deviations for
the capacity of the range of a 
random walk on $\Z^d$, in dimension $d\ge 5$, both in the upward and downward directions. 
The results are analogous to those we obtained for the volume 
of the range in two companion papers \cite{AS, AS19}. 
Interestingly, the main steps of the strategy
we developed for the latter apply
in this seemingly different setting, 
yet the details of the analysis are different. 
\\

MSC 2010 \emph{subject classifications.} Primary 60F05, 60G50.
\end{abstract}

\section{Introduction}\label{sec-intro}
We consider a simple random walk $(S_n)_{n\in \N}$ on $\Z^d$ starting
from the origin. The range of the walk between 
two times $k,n$ with $k\le n$, is denoted as 
$\cR[k,n]:=\{S_k,\dots,S_n\}$ with the shortcut $\cR_n=\cR[0,n]$. 
Its Newtonian capacity, denoted $\cc{\cR_n}$,
can be seen as the hitting probability of $\cR_n$
by an independent random walk starting from 
{\it far away} and properly normalized. Equivalently, using reversibility, 
it can be expressed as the sum of 
escape probabilities from $\cR_n$ by an independent random walk  
starting along the range. 
In other words, $\cc{\cR_n}$ is random and has the following representations: 
\begin{equation}\label{def-cap}
\cc{\cR_n}\ =\ \lim_{z\to\infty}\, 
\frac{\bP_{0,z} (\widetilde H_{\cR_n}<\infty\ |\ S)}{G(z)}=
\sum_{x\in \cR_n} \bP_{0,x} \big(\widetilde H^+_{\cR_n}=\infty\ |\ S\big),
\end{equation}
where $\bP_{0,z}$ is the law of two independent
walks $S$ and $\widetilde S$ starting at $0$ and $z$ respectively, $G(\cdot)$ is Green's function, and 
$\widetilde H_{\Lambda}$ (resp. $\widetilde H^+_{\Lambda}$) 
stands for the hitting (resp. return) time of $\Lambda$ 
by the walk $\widetilde S$.

In view of \reff{def-cap}, the study of the capacity of the range 
is intimately related to the question of estimating probabilities of 
intersection of random walks. This chapter has grown 
quite large, with several motivations from
statistical mechanics keeping the interest alive (see Lawler's
celebrated monograph \cite{Law91}). The last decade has witnessed revival interests both
after a link between uniform spanning trees and loop erased random walks
was discovered  
(see \cite{LawSW}, \cite{Hut} for recent results) and
after the introduction of random interlacements 
by Sznitman in \cite{S10} which mimic a random walk 
confined in a region of volume comparable to its time span.

The study of the capacity of the range of a random walk has
a long history. Jain and Orey \cite{JO} show that in 
any dimension $d\ge 3$, 
there exists a constant $\gamma_d\in [0,\infty)$, such that almost surely
\begin{equation}\label{cap.limit}
\lim_{n\to\infty}\frac 1n \cp(\RR_n) = \ \gamma_d,\quad\text{and}\quad
\gamma_d>0,\ \text{if and only if,}\ d\ge 5.
\end{equation}
The first order asymptotics is obtained in dimension $3$ in \cite{Chang},
where $\cp(\RR_n)$ scales like $\sqrt n$. Dimension $4$ is {\it the critical dimension}, 
and a central limit theorem with a non-gaussian limit is established in \cite{ASS19}.
In higher dimensions, a central limit theorem is proved 
in \cite{Sch19} for $d=5$, and in \cite{ASS18a} for $d\ge 6$. 

Here, we mainly study the downward deviations for
the capacity of the range in dimension $d\ge 5$,
in the moderate and large deviations regimes. 
We also establish a large deviations 
principle in the upward direction.
Our analysis is, as in our previous works \cite{AS,AS19}, related to the celebrated large
deviation analysis of the volume of the Wiener sausage 
by van den Berg, Bolthausen and den Hollander \cite{BBH01}.
The folding of the Wiener sausage, under squeezing its volume, became 
a paradigm of {\it folding}, with localization in a domain with holes 
of order one (the picture of a Swiss Cheese popularized in \cite{BBH01}). 
The variational formula for the rate function
was shown to have minimizers of different nature in $d=3$ and in $d\ge 5$
suggesting dimension-dependent optimal scenarii to achieve the deviation.
For the discrete analogue of the Wiener sausage, we established in 
\cite{AS, AS20a} 
some path properties confirming some observations of \cite{BBH01}.

\paragraph{Main results}
Our first result concerns the large and moderate deviations in dimension $7$ and higher. In this case, we obtain upper and lower bounds which are of the same order (on a logarithmic scale), and we cover (almost) the whole set of possible moderate deviations in the non-Gaussian regime.

\begin{theorem}\label{thm:d6}
Assume $d\ge 7$. There exist positive constants $\varepsilon$, 
$\underline \kappa$ 
and $\overline \kappa$ (only depending on the dimension),
such that for any $n^{\frac{d-2}{d}} \cdot \log n \le \zeta \le \varepsilon n$, and for $n$ large enough, 
\begin{equation}\label{dev.nongauss.7}
\exp\left(- {\underline\kappa}\cdot \zeta^{1-\frac{2}{d-2}} \right)\ \le\ 
\bP\left( \cc{\cR_n}-\bE[ \cc{\cR_n}]\le -\zeta\right)\ \le\ 
\exp\left(- {\overline\kappa}\cdot \zeta^{1-\frac{2}{d-2}} \right). 
\end{equation}
\end{theorem}
Recall that a central limit theorem is proved in 
\cite{ASS18a}, where we show in particular that $\var(\cc{\cR_n}) \sim \sigma^2 n$, for some constant $\sigma>0$. 
Our next result proves now a Moderate Deviation Principle in the Gaussian regime. 

\begin{theorem}\label{theo.Gaussian}
Assume $d\ge 7$. For any sequence $\{\zeta_n\}_{n\ge 0}$, 
satisfying $\lim_{n\to \infty} \zeta_n/\sqrt n = \infty$, 
and $\lim_{n\to \infty} \zeta_n (\log n) /  n^{\frac{d-2}d} = 0$, we have
\begin{equation}\label{limit-TCL}
\lim_{n\to \infty} \frac{n}{\zeta_n^2} \cdot 
\log \bP\left(\pm(\cc{\cR_n}-\bE[\cc{\cR_n}])> \zeta_n\right) 
=-\frac{1}{2\sigma^2}. 
\end{equation}
\end{theorem}

In dimension $5$, we obtain similar estimates, but we do not reach the Gaussian regime:  
\begin{theorem}\label{thm:d5} 
Assume $d=5$. There exist positive constants $\varepsilon$, 
$\underline \kappa$ and $\overline \kappa$, 
such that for any $ n^{5/7}\cdot \log n \le \zeta\le \varepsilon n$, and $n$ large enough,   
\begin{equation*}
\exp\left(- {\underline\kappa}\cdot(\frac{\zeta^2}{n})^{1/3}  \right)\ \le\
\bP\left( \cc{\cR_n}-\bE[ \cc{\cR_n}]\le -\zeta\right)\ \le\ 
\exp\left(- {\overline\kappa}\cdot (\frac{\zeta^2}{n})^{1/3} \right).
\end{equation*}
\end{theorem}
\begin{remark}\label{rem-dev-d5}
\emph{
In $d=5$, the variance of $\cp(\cR_n)$ is of order $n\log n$, 
\cite{Sch19}. Thus, the moderate deviations should go
from a gaussian regime with a speed
of order $\zeta^2/(n\log n)$, to a large deviation regime with
a speed of order $(\zeta^2/n)^{1/3}$, and
with a transition occurring for $\zeta$ of order $\sqrt {n}(\log n)^{3/4}$. 
For an explanation of the exponent $5/7$ which limits us here, 
see Remark \ref{rem.5}.  
Note that in the case of the volume of the range,
a similar transition has been established 
by Chen \cite{Chen} in dimension $3$, and by the authors in $d\ge 5$ in 
the companion paper \cite{AS19}.}
\end{remark}

\begin{remark}
\emph{In dimension $6$ our result is less precise. One can only show that 
\begin{equation*}
\exp\left(- {\underline\kappa} \cdot \zeta^{1/2} \right)\ \le\ 
\bP\left( \cc{\cR_n}-\bE[ \cc{\cR_n}]\le -\zeta\right)\ \le\ 
\exp\left(- \frac{{\overline\kappa}}
{\log (n/\zeta)}\cdot \zeta^{1/2} \right).
\end{equation*}
}
\end{remark}

Our next result provide path properties of the trajectory 
under the constraint of moderate deviations. To state it, 
one needs more notation. 
For $r>0$, and $x\in \Z^d$, set 
$$
Q(x,r):=[x-\frac r2,x+\frac r2)^d\cap \Z^d. 
$$
Given $\Lambda\subseteq \Z^d$, and $n\ge 0$, 
let $\ell_n(\Lambda)$ be the 
time spent in $\Lambda$ before time $n$. 
For $\rho \in (0,1]$, and $r,n$ positive integers, we let 
\begin{equation}\label{def-CV}
\C_n(r,\rho):= \{x\in r\Z^d\, :\, \ell_n(Q(x,r))\ge \rho r^d\},
\qquad\text{and }\qquad
\V_n(r,\rho):=\bigcup_{x\in \C_n(r,\rho)} Q(x,r).
\end{equation}
Define also
for a sequence of values of deviation $(\zeta_n)_{n\ge 1}$, 
\begin{eqnarray*}
\rho_{\textrm{typ}} := \left\{ 
\begin{array}{ll}
\zeta_n^{5/3}/n^{7/3} & \text{if }d=5\\
\zeta_n^{-2/(d-2)} & \text{if }d\ge 7, 
\end{array}
\right. 
\quad 
\tau_{\textrm{typ}}:= \left\{ 
\begin{array}{ll}
n & \text{if }d=5\\
\zeta_n & \text{if }d\ge 7, 
\end{array}
\right. 
\quad \text{and}\quad  
\chi_d:= \left\{ 
\begin{array}{ll}
5/7 & \text{if }d=5\\
\frac {d-2}{d} & \text{if }d\ge 7. 
\end{array}
\right. 
\end{eqnarray*}

\begin{theorem}\label{thm:scen-5}
Assume $d=5$, or $d\ge 7$. 
There are positive constants $\alpha$, $\beta$, $\varepsilon$ and $C_0$, such that for any sequence $(\zeta_n)_{n\ge 1}$, satisfying 
$$n^{\chi_d}\cdot \log n \le \zeta_n\le \epsilon n,$$  
defining $(r_n)_{n\ge 1}$ by 
$$r_n^{d-2}\rho_{\textrm{typ}} = C_0\log n,$$
one has 
\begin{equation}\label{result.path}
\lim_{n\to\infty}
\mathbb \bP\left(\ell_n(\V_n(r_n,\beta\rho_{\textrm{typ}}))
\ge \alpha\, \tau_{\textrm{typ}} \mid  \cc{\cR_n}-\bE[ \cc{\cR_n}]
\le -\zeta_n\right) = 1.
\end{equation}
Moreover, there exists $A>0$, such that 
\begin{equation}\label{result.capacite}
\lim_{n\to\infty}
\mathbb \bP\left(\cp(\V_n(r_n,\beta \rho_{\textrm{typ}})) \le A|\V_n(r_n,\beta \rho_{\textrm{typ}})|^{1-2/d} \mid  \cc{\cR_n}-\bE[ \cc{\cR_n}]
\le -\zeta_n\right) = 1.
\end{equation}
\end{theorem}
Theorem~\ref{thm:scen-5} provides some information on 
the density the random walk has to realize in order to achieve
the deviation. We obtain that $\V_n(r_n,\beta \rho_{\textrm{typ}})$ is typically ball-like, 
in the sense that its capacity is of the order of its volume 
to the power $1-2/d$, as it is the case for Euclidean balls.


The final result concerns the upward deviations. 
Our decomposition \reff{decomp-ASS} allows us
to adapt the argument of Hamana and Kesten, \cite{HK},
written for the volume of the range of a random walk. 

\begin{theorem}\label{theo:upward}
Assume $d\ge 5$. The following limit exists for all $x>0$: 
\begin{equation*}
\psi_d(x):=-\lim_{n\to\infty} \ \frac{1}{n}\log \bP\big(
\cc{\cR_n}\ge  n\cdot x\big).
\end{equation*}
Furthermore, there exists a constant $\gamma_d^*>\gamma_d$ (defined in \reff{cap.limit}), 
such that the function $\psi_d$ is continuous and convex on $[0,\gamma_d^*]$, increasing on $[\gamma_d,\gamma_d^*]$, and satisfies 
\begin{equation*}
\psi_d(x)  \left\{\begin{array}{ll}
 = 0 & \text{if }x\le \gamma_d \\
\in (0, \infty) & \text{if } x\in (\gamma_d,\gamma_d^*]\\
=\infty & \text{if }x>\gamma_d^*.
\end{array}
\right. 
\end{equation*}
\end{theorem}
We also obtain Gaussian upper bounds (up to a logarithmic factor)    
in the regime of moderate deviations, see Proposition \ref{cor:upward}.

\paragraph{Our approach to downward deviations.}
The cornerstone of our approach is a decomposition formula obtained
in \cite{ASS18b}: 
\begin{equation}\label{decomp-ASS}
\forall A,B\quad\text{finite sets of }\Z^d,\qquad\quad
\cc{A\cup B}=\cc{A}+\cc{B}-\chi_\CC(A,B),
\end{equation}
where $\chi_\CC(A,B)$ called {\it the cross-term} has a nice expression.
In this work, the decomposition \reff{decomp-ASS} allows us to
follow a simple approach devised in \cite{AS}, and later improved in \cite{AS19}, to study 
downward deviations for the volume of the range in dimensions $d\ge 3$. 
We partition the time-period of length $n$ into
intervals of length $T\le n$, and by iterating \eqref{decomp-ASS} appropriately one can write our functional of the range,
$\cc{\cR_n}$, as a sum of i.i.d. terms minus a certain sum of cross-terms
of the form $\chi_\C(\cR_{iT},\cR[iT,(i+1)T])$, with $i$ going from $1$ to  $\lfloor n/T\rfloor $. 
The so-called corrector, is the sum of these cross-terms that we integrate over $\cR[iT,(i+1)T]$. 
We then show that for some appropriate
time-scale $T$ it is this corrector which is responsible for (most of)  
the deviations.
The final step is to estimate the cost for such deviations. This analysis is similar to the corresponding one  
for the volume of the range that we performed in \cite{AS19}, 
but it also requires some new ingredients, in particular Lemmas \ref{lem-HD}, \ref{lem-asympt} and \ref{lem:rear}.

On the other hand the proof of Theorem \ref{theo.Gaussian} relies on the following estimate, similar to the result for the intersection of two ranges that was obtained in \cite{AS20c}: 
first we observe that $\chi_\C(A,B)$ is bounded above by (twice) 
another functional $\widetilde {\chi}(A,B)$, defined for any
$A,B\subseteq \Z^d$ by
$$\widetilde{\chi}(A,B) := \sum_{x\in A}\sum_{y\in B} 
\bP_x\big(H^+_{A}=\infty \big)\cdot G(x-y)\cdot
\bP_y\big(H^+_{B}=\infty \big). $$ 
We then show that for some $\kappa>0$, 
if $\cR_\infty$ and $\widetilde \cR_\infty$ are the ranges of 
two independent walks, 
\begin{equation}\label{stretch-chi}
\bE\left[\exp\left(\kappa\cdot \widetilde 
\chi(\cR_\infty, \widetilde \cR_\infty)^{1-\frac{2}{d-2}}\right)\right] <\infty.
\end{equation}

\paragraph{Heuristics.} 
We use the sign $\approx$ to express 
that two quantities are {\it of the comparable order} (which here will have a deliberately vague meaning, and precise statements come later). 
As already mentioned, the first
step in this work is a simple decomposition
for the capacity of a union of sets in term of a cross-term 
\begin{equation}\label{heur-1}
\chi_\CC(A,B)\ \approx\ 2\!\sum_{x\in A}\sum_{y\in B} 
\bP_x\big(H^+_{A}=\infty \big)\cdot G(x-y)\cdot
\bP_y\big(H^+_{B}=\infty \big),
\end{equation}
see \reff{decomp-1} and \reff{decomp-2} for
a precise expression. 
The key phenomenon responsible for producing a small
capacity for the range of a random walk is {\it an increase
of the cross-term on an appropriate scale}. In other words, the
walk {\it folds} into a ball-like domain in order to increase
some {\it self-interaction} captured by the cross-term.
Now to be more concrete, 
let us divide the range $\cR[0,2n]$ into two subsets
$\cR[0,n]$ and $\cR[n,2n]$. Let us call, for simplicity $\cR_n^1=\cR[0,n]-S_n$,
and $\cR_n^2=\cR[n,2n]-S_n$ the two subranges translated by $S_n$ so that they
become independent. By \eqref{heur-1} and translation invariance of the capacity we see that  
\[
\cc{\cR[0,2n]}=\cc{\cR_n^1}+\cc{\cR_n^2}-\chi_\CC(\cR_n^1,\cR_n^2).
\]
Now, assume that both walks stay inside a ball of radius $R$ a time
of order $\tau\le n$, and are unconstrained afterward. 
Thus, under the strategy we mentioned,
\begin{equation}\label{heur-4}
\begin{split}
\chi_\CC(\cR_n^1,\cR_n^2)\ \approx\  & G(R)\times \cc{\cR_\tau^1}
\times \cc{\cR_\tau^2}+\cO\big(G(\sqrt n)n^2\big)\\
\ \approx\ & G(R)\big( \min(\tau,R^{d-2})\big)^2+\cO\big(n^{\frac{6-d}{2}}\big).
\end{split}
\end{equation}
The term $\cO\big(G(\sqrt n)n^2\big)$ appears if $\tau$ is smaller than $n$,
and accounts for the unconstrained contribution to the cross-term.
In obtaining \reff{heur-4}, we have used that if $\cR_\tau^1$ and $\cR_\tau^2$ are inside
a ball of radius $R$, then their capacity is bounded by the
capacity of the ball, which is of order $R^{d-2}$, as well as by
their volume bounded by $\tau$.
Thus, it is useless to consider $\tau$ larger than $R^{d-2}$,
since then $\tau$ no more affects the cross-term and increasing $\tau$ (or decreasing $R$ below $\tau$) only makes the strategy more costly.
Now a deviation of order $\zeta$ is reached if 
\begin{equation}\label{heur-2}
\frac{1}{R^{d-2}} \tau^2\approx \zeta.
\end{equation}
Recall that the cost of being localized a time $\tau$ in a ball of radius $R$ is of order $\exp(-\tau/R^2)$ (up to a constant in the exponential). 
So we need to find a choice of $(\tau,R)$ which minimizes this cost under the
constraint \reff{heur-2}. In other words one needs to maximize $\sqrt \zeta\cdot R^{(d-6)/2}$.
This leads to two regimes.
\begin{itemize}
\item When $d=5$, $R$ (and then $\tau$) is as large as possible. So, $\tau=n$ and
$R^{d-2}=n^2/\zeta$ by \reff{heur-2}. The strategy is time homogeneous for any 
$\zeta$!
\item When $d\ge 7$, then $\tau$ is as small as possible,
that is $\tau=R^{d-2}=\zeta$. The strategy is time-inhomogeneous.
\end{itemize}
When $d=6$, the strategy remains unknown, but the cost should
be of order $\exp(-\sqrt \zeta)$.

\paragraph{Application to a polymer melt.} 
The model of random interlacements, introduced by Sznitman \cite{S10}, is roughly speaking 
the union of the ranges of trajectories obtained by a Poisson
point process on the space of doubly infinite trajectories, and is such that
the probability of avoiding a set $K$ is $\exp(-u\cdot \cc{K})$, where
$u>0$ is a fixed parameter.
With this in mind, let us consider the following model of polymer
among a polymer melt interacting by exclusion. 
We distinguish one polymer, a simple random walk, interacting with
a cloud of other random walk trajectories modeled by
random interlacements which we call for short {\it the melt}. 
The interaction is through exclusion:
the walk and the melt do not intersect. When integrating
over the interlacements law, the measure on the walk with the effective
interaction has a density proportional to $\exp(-u \cdot \cc{\cR_n})$, with respect to the law of a simple random walk. 

As a corollary of our deviation estimates, one can address some issues
on this polymer. 
Since this follows in the same way as the study of the Gibbs 
measure tilted by the volume of the range was a corollary of \cite{AS}, we 
repeat neither the statements corresponding to Theorem 1.8 of \cite{AS},
nor the proofs here. 
The simplest and most notable difference with the latter theorem is that
the proper scaling of the intensity parameter $u$ which provides a phase
transition is when it is of order $n^{-2/(d-2)}$ in dimension $d\ge 5$.
Thus, one would consider the polymer partition function as a function
of $u\in \R^+$
\[
Z_n(u)=\bE\Big[\exp\big(-\frac{u}{n^{2/(d-2)}}(
\cc{\cR_n}-\bE[\cc{\cR_n}])\ \big)\Big].
\]
Theorem 1.8 of \cite{AS} is true, here also after the drop in dimension
is performed, and establishes the existence of a phase transition 
as one tunes $u$.
On the other hand, considering the quenched model, where the random interlacements is given a typical realization, 
is an interesting open problem, beyond
the reach of the present techniques.

\paragraph{Organization.} The paper is organized as follows. In the next section, we recall some basic estimates on the random walk, the capacity, and the range that we will need.  Section~\ref{sec-martin}, and more precisely Proposition~\ref{prop.corrector} 
makes the link between downward deviations for the capacity and upward deviations of a corrector. 
The corrector itself
is studied in Section~\ref{sec-UB}, where we prove the upper bounds in 
Theorems~\ref{thm:d6} and \ref{thm:d5}, as well as Theorem \ref{thm:scen-5}. 
In Section \ref{sec-LB}, we prove the lower 
bounds in Theorems~\ref{thm:d6} and~\ref{thm:d5}.
The proof of Theorem \ref{theo.Gaussian} is done in Section \ref{sec.Gaussian}. Finally, we prove Theorem~\ref{theo:upward} concerning
the upward deviations in Section~\ref{sec-HK}. 

\section{Preliminaries}\label{sec-prelim}
\subsection{Further notation}
For $z\in \Z^d$, $d\ge 5$, we denote by $\bP_z$ the law of the simple random walk starting from $z$, and simply write it $\bP$ when $z=0$. We let 
$$G(z):=\mathbb E\left[\sum_{n=0}^\infty {\bf 1}\{S_n=z\}\right],$$
be the Green's function. 
It is known (see \cite{Law91}) that for some positive constants $c$ and $C$, 
\begin{equation}\label{Green}
\frac{c}{\|z\|^{d-2}+1}\ \le \ G(z)\ \le  \ \frac{C}{\|z\|^{d-2}+1},\quad \text{for all }z\in \Z^d,
\end{equation}
with $\|\cdot \|$ the Euclidean norm. 
We also consider for $T>0$, and $z\in \Z^d$,
$$G_T(z):=\mathbb E\left[\sum_{n=0}^T {\bf 1}\{S_n=z\}\right].$$
In particular for any $z\in \Z^d$, and $T\ge 1$,  
\begin{equation}\label{inGT}
\mathbb P(z\in \cR_T) \le G_T(z).
\end{equation}
For $A\subset \Z^d$, we denote by $|A|$ the cardinality of $A$, and by 
$$H_A:=\inf\{n\ge 0\, :\, S_n\in A\},\quad \text{and}\quad H_A^+:= \inf\{n\ge 1\, :\, S_n\in A\},$$
respectively the hitting time of $A$ and the first return time to $A$.

We also need the following well known fact, see \cite{Law91}. 
There exists a constant $C>0$, such that for any $R>0$ and $z\in \Z^d$,  
\begin{equation}\label{hit}
\bP_z\left(\inf_{k\ge 0}\|S_k\|\le R\right) \le C\cdot  \left(\frac{R}{\|z\|}\right)^{d-2}.
\end{equation}

\subsection{On the capacity}\label{sec-capa}
The capacity of a finite subset $A\subset \Z^d$, with $d\ge 3$, is defined by 
\begin{equation}\label{cap.def}
\cp(A):=\lim_{\|z\|\to \infty} \, \frac{1}{G(z)} \bP_z(H_A<\infty).
\end{equation}
It is well known, see Proposition 2.2.1 of \cite{Law91}, 
that the capacity is monotone for inclusion:
\begin{equation}\label{cap.mon}
\cp(A)\le \cp(B), \quad \text{for any }A\subset B,
\end{equation}
and satisfies the sub-additivity relation
\begin{equation}\label{cap.subadd}
\cp(A\cup B)\le \cp(A) +\cp (B) - \cp(A\cap B),\quad \text{for all }A,B\subset \Z^d.
\end{equation}
Another equivalent definition of the capacity is the following 
(see (2.12) of  \cite{Law91}). 
\begin{equation}\label{cap.def2}
\cp(A) =\sum_{x\in A} \bP_x(H_{A^+}=\infty).
\end{equation}
In particular it implies that 
\begin{equation}\label{cap.card}
\cp(A) \le |A|, \quad \text{for all }A\subset \Z^d.
\end{equation}
The starting point for our decomposition is the definition \eqref{cap.def} 
of the capacity in terms of a hitting time. 
It implies that for any two finite subsets $A,B\subset \Z^d$, 
\begin{equation}\label{decomp-1}
\cc{A\cup B}=\cc{A}+\cc{B}-\chi_\C(A,B),
\end{equation}
with 
\begin{equation*}
\chi_\C(A,B):=\lim_{z\to\infty} \frac{1}{G(z)}
\bP_z\big(\{H_{A}<\infty\}\cap \{H_{B}<\infty\}\big).
\end{equation*}
In particular by \eqref{cap.def} and the latter formula, one has 
\begin{equation}\label{borne.cap.AUB}
0\le \chi_\C(A,B) \le \min(\cp(A),\cp(B)).
\end{equation}
Now, we have shown in \cite{ASS19} that
\begin{equation}\label{decomp.chiC}
\chi_\C(A,B)=\chi(A,B) + \chi(B,A) -\varepsilon(A,B),
\end{equation}
with 
\begin{equation}\label{decomp-2}
\chi(A,B)=\sum_{x\in A}\sum_{y\in B} 
\bP_x\big(H^+_{A\cup B}=\infty \big)\cdot G(x-y)\cdot
\bP_y\big(H^+_{B}=\infty \big),
\end{equation}
and,
\begin{equation}\label{decomp-2bis}
0\le \varepsilon(A,B)\le \cc{A\cap B}\le |A\cap B|,
\end{equation}
where the last inequality follows from \eqref{cap.card}.

We will need some control on the speed of convergence in \eqref{cap.limit}. 
\begin{lemma}\label{lem.exp.cap}
Assume $d\ge 5$. One has 
\begin{equation*}
\left| \mathbb E[\cp(\cR_n)] - \gamma_d n \right| = \mathcal O(\psi_d(n)),
\end{equation*}
with 
\begin{equation*}
\psi_d(n) = \left\{
\begin{array}{ll} 
\sqrt n & \text{if }d=5\\
\log n & \text{if }d=6\\
1 & \text{if }d\ge 7.
\end{array}
\right. 
\end{equation*} 
\end{lemma}
\begin{proof}
By \eqref{decomp-1}, \eqref{decomp.chiC}, \eqref{decomp-2}, and \eqref{decomp-2bis} one has the rough lower bound:  
\begin{equation}\label{decomp.lemma}
\cp(\RR_{n+m}) \ge \cp(\RR_n) + \cp(\RR[n,n+m]) - 2\sum_{k=0}^n \sum_{\ell = n}^{n+m} G(S_k -S_\ell),
\end{equation}
for any integers $n,m\ge 1$ (a better inequality will be used later, but this one is enough here). 
Then one concludes exactly as in \cite{AS2}, using \eqref{cap.subadd}, Hammersley's  lemma  and Lemma 3.2 in \cite{ASS18a}, 
which controls the moments of the error term in the right-hand side of \eqref{decomp.lemma}. 
For the details, we refer to the proof of (1.13) in \cite{AS2}.   
\end{proof}

The next result provides some useful bounds on the variance of the capacity of the range, which were obtained in \cite{Sch19} in case of dimension $5$, and in \cite{ASS18a} in higher dimension. 
\begin{proposition}\label{lem.var.cap}
One has, 
\begin{eqnarray*}
\var(\cp(\RR_n)) =\left\{
\begin{array}{ll} 
\cO( n \log n) & \text{if }d=5\\
\cO(n) & \text{if }d\ge 6.
\end{array}
\right. 
\end{eqnarray*}
\end{proposition}
\begin{remark}
\emph{
Actually sharp asymptotics are known: in dimension $5$, one has $\var(\cp(\RR_n))\sim \sigma_5 n\log n$, and in higher dimension $\var(\cp(\RR_n))\sim \sigma_dn$, for some positive constant $(\sigma_d)_{d\ge 5}$, see respectively \cite{Sch19} and \cite{ASS18a}.  }
\end{remark}

As a consequence of the previous results one can obtain Gaussian type upper bounds for the moderate deviations in the upward deviations. 
\begin{proposition}\label{cor:upward} 
There exist positive constants $(c_d)_{d\ge 5}$, such that for any $n\ge 2$, and $\zeta>0$, 
\begin{eqnarray*}
\bP\big(\cc{\cR_n} - \bE[\cc{\cR_n}]\ge  \zeta \big) \le \left\{
\begin{array}{lll}
\exp\left(-c_5\cdot \frac{\zeta^2}{n(\log n)^3}\right) &  \text{if }d=5\\
\exp\left(-c_6\cdot \frac{\zeta^2}{n(\log \log n)^2}\right) & \text{if }d=6\\
\exp\left(-c_d\cdot \frac{\zeta^2}n\right) & \text{if }d\ge 7.
\end{array}
\right.
\end{eqnarray*}
\end{proposition}
\begin{proof}
For simplicity let us concentrate on the proof when $d=5$. 
We will explain at the end the necessary modifications to the proof when $d\ge 6$. 
Note first that one can always assume that $\zeta$ is smaller than $n/2$. 
We use now \eqref{decomp-ASS} repeatedly along a dyadic decomposition of $\{0,\dots,n\}$. This gives for $L\ge 1$, with $m_L:=\lfloor n/2^L\rfloor$, 
$$\cc{\cR_n} = \sum_{i=1}^{2^L}  \cc{\cR^{(i)}_{m_L}} - \sum_{\ell = 1}^L \Sigma_\ell,$$
where the $\cR^{(i)}_{m_L}$ are consecutive pieces of the range of length either $m_L$ or $m_L+1$, and 
$$\Sigma_\ell := \sum_{j=1}^{2^{\ell - 1}} \chi_C(\RR^{(2j-1)}_{m_\ell},\RR^{(2j)}_{m_\ell}),$$
with similar notation as above, in particular $m_\ell = \lfloor n/2^\ell\rfloor$. 
Thus, 
\begin{align}\label{proof.upward}
\nonumber \bP\big(\cc{\cR_n}-\bE[\cc{\cR_n} ]> \zeta \big)  \le\  & 
  \bP\left(\sum_{i=1}^{2^L}  \cc{\cR^{(i)}_{m_L}}-\bE[\cc{\cR^{(i)}_{m_L}}] >\frac \zeta 2\right)\\
& + \sum_{\ell = 1}^{L} \bP\left( \bE[\Sigma_\ell] - \Sigma_\ell > \frac{\zeta}{2L}\right).
\end{align}
We fix now $L$, such that $n/\zeta \le m_L  \le 2n/\zeta$. 
The first term in \eqref{proof.upward} is ruled out using Bernstein's inequality and Proposition \ref{lem.var.cap}, which give 
for some constant $c>0$. 
\begin{eqnarray}\label{proof.upward2} 
\bP\left(\sum_{i=1}^{2^L}   \cc{\cR^{(i)}_{m_L}}-\bE[\cc{\cR^{(i)}_{m_L}}] >\zeta/2\right) \le \exp\left(-c \frac{\zeta^2}{n \log m_L}\right).
\end{eqnarray}
Concerning the sum in \eqref{proof.upward}, note first that by Lemma \ref{lem.exp.cap}, one has 
$$\bE[\Sigma_\ell] = \cO(2^{\ell/2} \sqrt n) ,$$
for any $\ell \ge 1$. Therefore, one can assume that $\ell$ is such that $2^{\ell/2} \sqrt{n} > c\zeta/L$, for some constant $c>0$, for otherwise the corresponding probability is zero.  
For such $\ell$ one has by using standard concentration results (see Theorem 4.4. in \cite{CL}): 
$$\bP\left( \bE[\Sigma_\ell] - \Sigma_\ell > \frac{\zeta}{2L}\right) \le \exp\left(-c\frac{(\zeta/L)}{ \sqrt{m_\ell} +L n (\log m_\ell)/\zeta}\right) \le \exp\left(-\frac{c\zeta^2}{n(\log n)^3}\right),$$
which completes the proof in case $d=5$. In case $d\ge 6$, the variance is linear. So first, the term $\log m_L$ can be removed in \eqref{proof.upward2}. 
Moreover, in case $d=6$, one has $\bE[\Sigma_\ell ] = \cO(2^\ell \log n)$, and thus only the $\ell$'s such that $\zeta \ge 2^\ell \ge c\zeta/\log n$ need to be considered.   
There are order $\log \log n$ such integers, and for each of them one has by the same argument as above, 
$$ \bP\left( \bE[\Sigma_\ell] - \Sigma_\ell > \frac{\zeta}{C\log \log n}\right)\le \exp(-c\zeta^2/(\log \log n)^2),$$
which concludes the proof in case $d=6$. The case $d\ge 7$ is similar, since this time $\bE[\Sigma_\ell ] = \cO(2^\ell)$, and thus there are only a bounded number of integers $\ell$'s that need to be considered.   
\end{proof}

\section{Transfer of downward deviations to the corrector}\label{sec-martin}
The possibility of establishing the heuristic picture
described in the introduction stems from writing 
the capacity of a union of sets as a sum of capacities
and a cross-term. The latter though typically small is nonetheless
responsible for the fluctuations. 
Iterating this decomposition leads to an expression of the capacity of the range as a sum of i.i.d. terms minus a sum of cross-terms. 
The so-called corrector is obtained by summing appropriate conditional expectations of these cross-terms.

Our first result in this section, Lemma \ref{lem.corrector}, provides an explicit expression for (what turns out to be an upper bound for) this corrector in terms of a  sum of convoluted 
Green's functions taken along the trajectory and 
weighted by escape probability terms. We then recall a result from \cite{AS19}, which relates the deviations of the capacity to those of the corrector, which we state here as Proposition \ref{prop.corrector}.

Thus, the strategy is similar to the one used to treat
downward deviations for the range  
developed in \cite{AS19}. However
the form of the corrector is slightly different. Roughly it involves a convolution of Green's function with itself together multiplied by escape 
probability terms, where in \cite{AS} only Green's function appeared.

A detailed analysis of this corrector is 
carried out in Sections \ref{sec-d7} and \ref{sec-d5}. 
Before we can state precisely the result, 
some preliminary notation is required.

For $I\subset \N $, we write $\cR(I):=\{S_k,\ k\in I\}$, 
for the set of visited sites during times $k\in I$. 
Since for any two intervals $I,J\subset \N$, one has $\cR(I\cup J)=\cR(I)\cup \cR(J)$, \reff{decomp-1} gives 
\begin{equation}\label{decomp-3}
\cc{\cR(I\cup J)}= \cc{\cR(I)}+\cc{\cR(J)}-\chi_\C\big(\cR(I),\cR(J)\big).
\end{equation}
Next, given two sets $A$ and $B$, their symmetric difference is defined as $A\Delta B:=(A\cap B^c) \cup (B\cap A^c)$. 
Note in particular that for any $I,J\subset \N$, one has $\cR(I)\Delta \cR(J)\subset \cR(I\Delta J)$. 
Moreover, it follows from \eqref{cap.mon}, \eqref{cap.subadd} and \eqref{cap.card} that for any $A,B\subset \Z^d$, 
$$|\cp(A)-\cp(B)|\le \cp(A\Delta B)\le |A\Delta B|. $$ 
Applying this inequality to ranges on some intervals $I$ and $J$, we get 
\begin{equation}\label{diff.sym.ranges}
|\cp(\cR(I)) - \cp(\cR(J))|\le |I\Delta J |. 
\end{equation}

Now given some integer $T\le n$, we define for $j\ge 0$ and $\ell \ge 1$,  
$$I_{j,\ell}:=[j+(\ell-1)T,j+\ell T], \quad \text{and}\quad \widetilde I_{j,\ell}:=I_{j,1}\cup \dots\cup I_{j,\ell}.$$ 
It follows from \eqref{diff.sym.ranges} that almost surely 
\begin{equation}\label{decomp-6}
|\cp(\cR_n)- \frac{1}{T}\sum_{j=0}^{T-1} 
\cp(\cR(\widetilde I_{j,\lfloor n/T\rfloor} )) |\le T.
\end{equation} 
On the other hand, applying \eqref{decomp-3} recursively we obtain for any $j=0,\dots,T-1$, 
\begin{equation}\label{decomp-7}
\cp(\cR(\widetilde I_{j,\lfloor n/T\rfloor} ) ) = \sum_{\ell=1}^{ \lfloor n/T\rfloor}
\cp (\cR(I_{j,\ell})) - \sum_{\ell=1}^{\lfloor n/T\rfloor -1}
\chi_\C \big(\cR(\widetilde I_{j,\ell}),\cR(I_{j,\ell+1})\big) .
\end{equation}

Define now 
$$\chi_n(T):= \frac{1}{T}\sum_{j=0}^{T-1} \sum_{\ell=1}^{\lfloor n/T\rfloor -1}
\chi_\C \big(\cR(\widetilde I_{j,\ell}),\cR(I_{j,\ell+1})\big), $$
and note that  \eqref{decomp-6} and \eqref{decomp-7} give for any $T\le \zeta/2$,  
\begin{align}\label{dev-decomp}
\nonumber & \bP\left(\cp(\cR_n)-  \mathbb E[\cp(\cR_n)] \le -\zeta \right) \le  \bP\left( \frac{1}{T}\sum_{j=0}^{T-1} \cp(\cR(\widetilde I_{j,\lfloor n/T\rfloor} )) - \mathbb E[\cp(\cR(\widetilde I_{j,\lfloor n/T\rfloor} ))] \le -\frac {\zeta}{2}\right)  \\
&  \le  \bP\left( \frac{1}{T}\sum_{j=0}^{T-1} \sum_{\ell=1}^{ \lfloor n/T\rfloor}
\cp (\cR(I_{j,\ell})) - \mathbb E[\cp (\cR(I_{j,\ell}))] \le -\frac {\zeta}{4}\right)  + \bP\left(\chi_n(T) \ge \frac{\zeta}{4}\right). 
\end{align}
The first term on the right-hand side of \eqref{dev-decomp} is dealt with Bernstein's inequality and Proposition \ref{lem.var.cap}, which show that for any $\zeta> \frac{n\log n}{T}$, 
for some constant $c>0$. 
\begin{align}\label{dev-decomp.2}
\nonumber & \bP\left( \frac{1}{T}\sum_{j=0}^{T-1} \sum_{\ell=1}^{ \lfloor n/T\rfloor}
\cp (\cR(I_{j,\ell}))- \mathbb E[\cp (\cR(I_{j,\ell}))]  \le -\frac {\zeta}{4}\right) \\
& \le T \max_{0\le j\le T-1} 
\bP\left( \sum_{\ell=1}^{ \lfloor n/T\rfloor}
\cp (\cR(I_{j,\ell}))- \mathbb E[\cp (\cR(I_{j,\ell}))]  
\le -\frac {\zeta}{4}\right) \le T \exp(-c\frac{\zeta}{T}).
\end{align}
For the second term in the right-hand side of \eqref{dev-decomp}, we will use a general result of \cite{AS19}, which allows to compare the (moderate) deviations of $\chi_n(T)$ to those of its compensator, defined by 
\begin{equation}\label{def-xin*}
\xi_n^*(T):= \frac 1T \sum_{j=0}^{T-1} \sum_{\ell = 1}^{\lfloor n/T\rfloor-1} \bE\left[\chi_\C \big(\cR(\widetilde I_{j,\ell}),\cR(I_{j,\ell+1})\big)\mid \mathcal F_{j+\ell T}\right].  
\end{equation}
More specifically, Proposition 4.1 in \cite{AS19} (see also the proof of Corollary 4.2 there) shows that for some constant $c>0$, for any $\zeta>0$, 
\begin{equation}\label{dev-chinT}
\bP(\chi_n(T) \ge \frac {\zeta}{4}) \le \exp(-c\frac{\zeta}{T}) + \bP(\xi_n^*(T) \ge c\zeta ) , 
\end{equation}
(where here we use also that by \eqref{cap.card} and \eqref{borne.cap.AUB}, each term of the sum in the definition of $\chi_n(T)$ is bounded by $T$). 
We next define   
\begin{equation}\label{def-xin}
\xi_n(T):= \sum_{k=0}^n \sum_{x\in \cR_k} 
\bP_x\big(H^+_{\cR_k}=\infty\big)\cdot\frac{G\star G_T(x-S_k)}{T}. 
\end{equation} 

\begin{lemma}\label{lem.corrector}
One has, for any $n\ge 1$ and $1\le T\le n$,
$$\xi_n^*(T)\le 2\xi_n(T).$$ 
\end{lemma}
\begin{proof} By \eqref{decomp-2}, for any sets $A$ and $B$, 
$$\chi(A,B)\le \widetilde \chi(A,B):= \sum_{x\in A}\sum_{y\in B} 
\bP_x\big(H^+_A=\infty \big)\cdot G(x-y)\cdot
\bP_y\big(H^+_{B}=\infty \big). $$
Note that $\widetilde \chi$ is symmetric 
in the sense that $\widetilde \chi(A,B) = \widetilde \chi(B,A)$, for any $A,B$. 
Bounding the last probability term by one, we get 
$$\chi_\C(A,B)\stackrel{\eqref{decomp.chiC}}{\le} \chi(A,B) + \chi(B,A) \le 2 \overline \chi(A,B),
\quad\text{with }\quad
\overline \chi(A,B):=\sum_{x\in A} \sum_{y\in B} \bP_x\big(H^+_A=\infty \big)\cdot G(x-y).$$
Now for any $j,\ell$, the Markov property and translation invariance of the simple random walk give
\begin{align*}
\mathbb E\left[\overline \chi \big(\cR(\widetilde I_{j,\ell}),\cR(I_{j,\ell+1})\big)
\mid \mathcal F_{j+\ell T}\right]
& =\   \sum_{x\in \cR(\widetilde I_{j,\ell}) }    \mathbb P_x(H_{\cR(\widetilde I_{j,\ell})}^+=\infty) \sum_{y\in \Z^d} G(x-y) \cdot 
\mathbb P\big( y \in \cR (I_{j,\ell+1}) \mid  \mathcal F_{j+\ell T}\big) \\
& \stackrel{\eqref{inGT}}{\le} \,  \sum_{x\in \cR(\widetilde I_{j,\ell})} \mathbb P_x(H_{\cR(\widetilde I_{j,\ell})}^+=\infty) \cdot G\star G_T(x-S_{j+\ell T}), 
 \end{align*} 
and the lemma follows from the definition \eqref{def-xin} and \eqref{def-xin*} of $\xi_n(T)$ and $\xi_n^*(T)$ respectively.
\end{proof}

Combining \eqref{dev-decomp}, \eqref{dev-decomp.2}, \eqref{dev-chinT}, and Lemma \ref{lem.corrector} we obtain the main result of this section. 
\begin{proposition}\label{prop.corrector}
There exists a positive constant $c$, such that for any $n\ge 2$, $\zeta>0$, and $T\ge 1$ satisfying $T\le \zeta/2$, and $\zeta \ge \frac{n\log n}{T}$,  
$$\bP\left(\cp(\cR_n)- \mathbb E[\cp(\cR_n)] \le -\zeta \right) \le 2T\exp(-c\frac{\zeta}{T}) + \bP\left(\xi_n(T) \ge c\zeta\right).$$
\end{proposition}

\begin{remark}\label{rem.5}
\emph{In dimension $5$, the mean of $\xi_n(T)$ is of order $n/\sqrt T$. So the upper deviations for $\xi_n(T)$ start to decay only for 
$\zeta>n/\sqrt T$, and since on the other hand one needs to take $T$ at most of order $(\zeta n)^{1/3}$, to ensure the term $\exp(-c\zeta/T)$ to have the 
right decay, this imposes the condition $\zeta>n^{5/7}$. In particular the approach we have here has no chance to work up to the Gaussian regime. 
On the other hand in dimension $7$ and higher, the mean of $\xi_n(T)$ is of order $n/T$, and $T$ can be chosen of order $\zeta^{2/(d-2)}$, 
which only imposes the a priori condition $\zeta>n^{(d-2)/d}$, leaving a chance to cover entirely the non-Gaussian regime.}
\end{remark}

\section{Upper Bounds}\label{sec-UB}
We prove here the upper bounds in \eqref{dev.nongauss.7} and in Theorem \ref{thm:d5}, as well as Theorem \ref{thm:scen-5}. 
We start by some preliminaries, which shall be used as well in Section \ref{sec.Gaussian}, concerning the Gaussian regime. 

\subsection{Basic estimates}\label{sec-ineq}
For $r>0$, and $x\in \R^d$, we recall that 
$ Q(x,r):=[x-r/2,x+r/2)^d\cap \Z^d,$
and for simplicity $Q(r):=Q(0,r)$.

\begin{lemma}\label{lem-HD}
Assume that $d\ge 5$. 
There exists a constant $C_1>0$, such that for any $r\ge 1$, and any   
$\Lambda\subset Q(r)$, 
\begin{equation}\label{lem-HD.1}
\sum_{x\in \Lambda} \frac{1}{\|x\|^{d-4}+1}
\cdot \bP_x(H^+_\Lambda=\infty)\, \le\, C_1\, r^2. 
\end{equation}
\end{lemma}
\begin{proof}
Without loss of generality, one can assume $r\ge 2$. For $i\ge 0$, write 
$$\Lambda_i := \Lambda\cap \left( Q( r2^{-i})\bs Q(r2^{-i-1})\right),$$
and define  
$L:=\lfloor \log_2(r) \rfloor$. Then, for some positive constants $C_0$ and $C_1$, 
\begin{equation*}
\begin{split}
\sum_{x\in \Lambda} \frac{1}{\|x\|^{d-4}+1}
\cdot \bP_x(H^+_\Lambda=\infty) \ &
\le \ \sum_{i=0}^L \sum_{x\in \Lambda_i}
 \frac{1}{\|x\|^{d-4}+1}\cdot \bP_x(H^+_\Lambda=\infty)\\
& \le \ \sum_{i=0}^L \big(\frac{2^{i+1}}{r}\big)^{d-4}\cc{\Lambda_i }
\le\  \sum_{i=0}^L \big(\frac{2^{i+1}}{r}\big)^{d-4}\cc{Q(\frac{r}{2^i})}\\
&\le C_0\, \sum_{i=0}^L \big(\frac{2^{i+1}}{r}\big)^{d-4}\cdot 
\big(\frac{r}{2^{i}}\big)^{d-2} \le\ C_1 \, r^2.
\end{split}
\end{equation*}
\end{proof}

The second result we need is the following.

\begin{lemma}\label{lem-asympt}
Assume $d\ge 5$. There exists a constant $C_2>0$, such that for any $x\in \Z^d$, and any $T\ge 1$,
\begin{equation*}
\varphi_T(x):=\frac{G\star G_T(x)}{T}\ \le\ C_2\cdot \min\left(\frac{1}{1+\|x\|^{d-2}},
\frac{1}{T(1+\|x\|^{d-4})}\right).
\end{equation*}
\end{lemma}
\begin{proof}
First $G_T\le G$, so that $G\star G_T\le G\star G$, and
an elementary computation gives that $G\star G(x)\le C_2/ (1+\|x\|^{d-4})$, for all $x\in \Z^d$, and some $C_2>0$. This already proves one of the two desired bounds.

For the other one write, by definition of $G_T$, 
\begin{equation}\label{GstarGS}
G\star G_T(x) = \sum_{y\in \Z^d} G(x-y)G_T(y) = \sum_{k=1}^T \mathbb E[G(x-S_k)].
\end{equation}
Let $\tau$ be the hitting time of the cube $Q(x,2)$ for the walk starting at 0, 
and note that one can assume $\|x\|\ge 4$.  
Since $G$ is harmonic on $\Z^d\bs \{0\}$, we have for any $k\ge 0$, 
$\bE [G(x-S_{k\wedge \tau})]=G(x)$. This entails
\[
G(x)=\bE [\1\{\tau\ge k\}G(x-S_{k})]+
\bE [\1\{\tau<k\}G(x-S_{\tau})]\ge 
\bE [G(x-S_k)]-\bE[\1\{\tau<k\} G(x-S_k)]. 
\]
Now, we use that $G(x)$ is bounded by $G(0)$, so that the previous inequality gives  
\[
 \bE [G(x-S_k)] \le G(x)+G(0) \bP(\tau<\infty) \stackrel{\eqref{hit}}{\le} (1+CG(0))\cdot G(x),
\]
for some constant $C>0$. 
Injecting this in \eqref{GstarGS} and using \eqref{Green}, proves the second inequality. 
\end{proof}

Our last estimate requires some new notation. 
For a (deterministic) function $S:\N \to \Z^d$ (not necessarily to the nearest neighbor), 
and for any  $\K\subset \N$, we define for any $\Lambda\subset \Z^d$, 
$$\ell_\K(\Lambda):=\sum_{k\in \K} {\bf 1}\{S(k) \in \Lambda\}.$$ 

\begin{lemma}\label{lem:rear}
Assume $d\ge 3$. Let $S:\N \to \Z^d$, and $\K\subset \N$, be such that for some $\rho\in (0,1)$ and $r\ge 1$, 
$$\ell_\K(Q(x,r)) \le \rho r^d, \qquad \text{for all }x\in r\Z^d.$$
There exists a constant $C_3>0$ (independent of $\rho$, $r$, $S$, and $\K$), such that for any $z\in \Z^d$,
\begin{equation}\label{lem:rear.2}
\sum_{k\in \K} \frac{\1(\|S(k)-z\|\ge 2r)}{\|S(k)-z\|^{d-2}} \ \le\ C_3 \, \rho^{1-\frac 2d} \,  |\K|^{2/d}.
\end{equation}
\end{lemma}
\begin{proof}
We start by proving that for any $R\ge 2r$, and any $z\in \Z^d$, 
\begin{equation}\label{lem:rear.1}
\sum_{k\in \K} \frac{\1(2r \le \|S(k)-z\|\le R)}{\|S(k)-z\|^{d-2}} \ \le\ C_3 \, \rho \, R^2. 
\end{equation}
Consider a covering of the cube $Q(z,R)$ by a partition made of smaller cubes which are translates of  $Q(r)$, with centers in the set $z+r\Z^d$.  
For each $x\in z+r\Z^d$, with $x\neq z$, the contribution of the points $S(k)$ lying in $Q(x,r)$ to the sum we need to bound, is upper bounded (up to some constant) by $\rho r^d \cdot \|x-z\|^{2-d}$, and \eqref{lem:rear.1} follows as we observe that, for some constant $C>0$, 
$$\sum_{x\in z+r\Z^d} \frac{\1\{r\le \|z-x\|\le R\}}{\|z-x\|^{d-2}}  \le C \frac{R^2}{r^d}.$$

We then deduce \eqref{lem:rear.2}, by observing that by rearranging the points $(S(k))_{k\in \K}$, one can only increase the sum (at least up to a multiplicative constant) by assuming they are all in $Q(z,2(\frac{|\K|}{\rho})^{1/d})$, and still satisfy the hypothesis of the lemma. 
\end{proof}

\subsection{The sets $\mathcal K_n$}\label{sec.KnAn}
We  recall here our main tools from \cite{AS19}, which require some new notation. For $n\ge 0$, and $\Lambda\subseteq \Z^d$, define 
the time spent in $\Lambda$ by the walk up to time $n$ as 
$$
\ell_n(\Lambda):= \sum_{k=0}^n \1\{S_k\in \Lambda\}.
$$
Then given $\rho>0 $, $r\ge 1$, and $n\ge 1$, set 
\begin{equation}\label{def-K}
\K_n(r,\rho):=\{k\le n:\ \ell_n(Q(S_k,r))\ge \rho r^d\}.
\end{equation}
The following result is Theorem 1.5 and Proposition 3.1 from \cite{AS19}. 
\begin{theorem}[\cite{AS19}] 
\label{lem-AS}
There exist positive constants $C_0$ and $\kappa$, 
such that for any $\rho>0$, $r\ge 1$, and $n\ge 1$, satisfying  
\begin{equation}\label{hyp.r}
\rho\, r^{d-2} \ge C_0\log n,
\end{equation} 
one has for any $L\ge 1$, 
\begin{equation*}
\mathbb P\big(|\K_n(r,\rho)|\ge L\big)\, 
\le \, C_0 \exp\left(-\kappa\, \rho^{\frac 2d}\, L^{1-\frac 2d} \right).
\end{equation*}
Furthermore, for any $A>0$, there exists $\alpha>0$, such that 
\begin{equation*}
\mathbb P\big(|\K_n(r,\rho)|\ge L, \, \ell_n(\V_n(r,2^{-d}\rho)) \le \alpha L\big)\, 
\le \, C_0 \exp\left(-A \rho^{\frac 2d}\, L^{1-\frac 2d} \right).
\end{equation*}
\end{theorem}


\subsection{Dimension seven and larger}\label{sec-d7}
We assume here that $d\ge 7$, and fix the value of $T$ as  
\begin{equation}\label{def.T.7}
T:= \lceil \gamma \cdot \zeta^{\frac{2}{d-2}} \rceil,
\end{equation}
for some constant $\gamma\in (0,1)$ (depending on dimension $d$) that will be fixed later (in the proof of Theorem \ref{thm:scen-5} below).  
Under the event of moderate deviations considered here (when the capacity of the range up to time $n$  is reduced by an amount $\zeta$ from its mean value), 
the walk typically folds its trajectory a time of order $\zeta$, in a region of volume $\zeta^{d/(d-2)}$. Thus the typical density of the range in the folding region is  
$$\overline \rho := \zeta^{-\frac{2}{d-2}}.$$
Define $\rho_i$, $r_i$, and $L_i$, for $i\in \Z$, by  
$$\rho_i := 2^{-i} \cdot \overline \rho, \qquad r_i^{d-2}\cdot\rho_i = C_0\log n, \qquad \text{and}\qquad L_i:= \zeta \cdot 2^{\frac{2i}{d-2}},$$
with $C_0$ as in Theorem \ref{lem-AS}.  Define
$$N:=\lceil \frac{d-2}{2}\cdot \log_2(n/\zeta)\rceil, \quad \text{and} \quad M:= \lceil  \log_2(1 /\overline \rho )\rceil,$$ 
so that $n \le L_N \le 2 n$, and $1\le \rho_{-M} \le 2$. 
For $-M\le i\le N$, set  
\begin{equation*}
\widehat \K_i : = \K_n(r_i,\rho_i)\setminus \bigcup_{-M\le j<i} 
\K_n(r_j,\rho_j),  
\end{equation*} 
with the convention that $\widehat \K_{-M} = \K_n(r_{-M},\rho_{-M})$. 
Finally for $A>0$, $\delta>0$, and $I< \min(M,N)$, define
\begin{equation*}
\mathcal E(A,\delta, I):=\left( \bigcap_{-I\le i\le I} \left\{|\widehat \K_i|\le \delta L_i\right\}\right)\cap \left( \bigcap_{I<i\le N}  \left\{|\widehat \K_i|\le A L_i\right\}\right)\cap \left( \bigcap_{-M\le i< -I}  \left\{|\widehat \K_i|\le A L_i\right\}\right).
\end{equation*}
Our main result here is the following proposition. 
\begin{proposition}\label{prop.Anxin7}
For any $A>0$, there exist $\delta>0$ and $I\ge 0$, such that for any $n\ge 2$, and $n^{\frac{d-2}{d}}\cdot \log n \le \zeta\le n$, 
$$\mathcal E(A,\delta, I) \ \subseteq \ \{\xi_n(T) \le \zeta\}.$$
\end{proposition}
Before we give the proof of this proposition, let us show how it implies the upper bound in Theorem \ref{thm:d6}, as well as Theorem \ref{thm:scen-5} for dimension $7$ and higher, assuming for a moment the lower bound in Theorem \ref{thm:d6} (which will be proved later and independently in Section \ref{sec-LB}). 

\begin{proof}[Proof of Theorem \ref{thm:d6}: the upper bound] 
Note first that Proposition \ref{prop.Anxin7} and Theorem \ref{lem-AS} give 
$$\bP(\xi_n(T) > \zeta) \le \bP( \mathcal E(1,\delta, I)^c) \le C\exp(-c\zeta^{1-\frac{2}{d-2}}),$$
for some constant $c>0$, where $\delta$ and $I$ are those given by Proposition \ref{prop.Anxin7}, associated to $A=1$. Note also that by definition $T$ is of order $\zeta^{2/(d-2)}$, see \eqref{def.T.7}, and thus the above estimate together with 
Proposition \ref{prop.corrector} prove the upper bound in Theorem \ref{thm:d6}. 
\end{proof}

\begin{proof}[Proof of Theorem \ref{thm:scen-5}] 
Assume the lower bound in Theorem \ref{thm:d6}, and let us start with the proof of \eqref{result.path}. 
First choose $\gamma$ small enough in the definition \eqref{def.T.7} of $T$, so that 
conditionally on the event  of moderate deviations 
$MD(n,\zeta):=\{\cp(\cR_n)-\bE[\cp(\cR_n)]\le -\zeta\}$, the probability of the event $\{\xi_n(T) \le c\zeta\}$ goes to zero, with $c$ some appropriately chosen constant. Note that this is possible thanks to (the proof of) Proposition \ref{prop.corrector}. Then choose $A$ large enough, so that conditionally on $MD(n,\zeta)$, 
the probability of any of the events $\{|\widehat \K_i|>AL_i\}$, for $i\in \Z$, goes to zero (which is always possible thanks to Theorem \ref{lem-AS}), where implicitly $\zeta$ is replaced by $c\zeta$ in the definition of these events. 
Then Propositions \ref{prop.corrector} and \ref{prop.Anxin7} show that conditionally on $MD(n,\zeta)$, one of the events 
$\{|\widehat \K_i|>\delta L_i\}$, with $-I\le i \le I$, holds with probability going to $1$ (where $\delta$ and $I$ are given by Proposition \ref{prop.Anxin7}), 
and \eqref{result.path} follows from the second part of Theorem \ref{lem-AS}. 

Finally, the characterization of the capacity in \reff{result.capacite}, 
is a simple consequence of a general result of \cite{AS20b}, namely
(1.15) of Theorem 1.5, once we know \reff{result.path}.
\end{proof}

\begin{proof}[Proof of Proposition \ref{prop.Anxin7}]
Let $\widehat \K_{N+1}$ be such that
\begin{equation}\label{eq.Ki}
\widehat \K_{N+1}:=\{0,\dots,n\}\bs \bigcup_{-M\le i \le N} \widehat \K_i.
\end{equation}
Now, we decompose $\xi_n(T)$ over the various $\widehat \K_i$.
By \eqref{def-xin}, for any $I\le \min(-M,N)$, 
$$
\xi_n(T) \le \Sigma_1 + \Sigma_2 + 2 \Sigma_3 + 2 \Sigma_4+2\Sigma_5,
$$
where (note that $r_i\le \sqrt T$ when $i\le 0$, and
$\varphi_T(z)=\frac 1T G\star G_T(z)$ is defined in Lemma \ref{lem-asympt}) 
$$
\Sigma_1 : = \sum_{i=-I}^{N+1} \sum_{k\in \widehat \K_i} \sum_{x\in \cR_k} \varphi_T(x-S_k) \bP_x(H^+_{\cR_k}=\infty)\cdot  \1\{x\in Q(S_k,r_{i-1})\},$$
$$
\Sigma_2 : = \sum_{i=-M}^{-I} \sum_{k\in \widehat \K_i} \sum_{x\in \cR_k} \varphi_T(x-S_k) \bP_x(H^+_{\cR_k}=\infty)\cdot  \1\{x\in Q(S_k,\sqrt T)\},$$
$$
\Sigma_3:= \sum_{i=-M}^{N+1} \sum_{j=i}^{N+1} \sum_{k\in \widehat \K_i} \sum_{k'\in \widehat \K_j}  \varphi_T(S_{k'}-S_k) \cdot  \1\{S_{k'}\in Q(S_k,r_{j-1})\setminus Q(S_k,r_{i-1})\},$$
$$
\Sigma_4:= \sum_{i=-M}^{0} \sum_{j=i}^{0} \sum_{k\in \widehat \K_i} \sum_{k'\in \widehat \K_j}  \varphi_T(S_{k'}-S_k) \cdot  \1\{S_{k'}\notin Q(S_k,\sqrt T)\},$$
$$
\Sigma_5:= \sum_{i=-M}^{N+1} \sum_{j=\max(i,0)}^{N+1} \sum_{k\in \widehat \K_i} \sum_{k'\in \widehat \K_j}  \varphi_T(S_{k'}-S_k) \cdot  \1\{S_{k'}\notin Q(S_k,r_{j-1})\},$$

Note that the third term $\Sigma_3$ is not included in $\Sigma_1$ and $\Sigma_2$, since in these last two terms we sum over points of the space, not over time indices. This is important since one important tool used to control them is Lemma \ref{lem-HD}.

Now assume that $\mathcal E(A,\delta, I)$ holds, and let us bound $\Sigma_1$ first. 
For $-I\le i\le N+1$, define 
$$J(i) = - i + \lfloor  \frac{d-4}{2} \log_2(T)-\frac d2\log_2 (\log n) - \frac{d-2}{2} h\rfloor , $$
with $h$ some positive constant to be chosen later, so that for any $-I\le i\le N+1$, 
$$\frac {L_i \cdot r_{J(i)}^2}{T} \le C 2^{-h} \cdot \frac{\zeta}{\log n},$$
for some constant $C>0$ (that might change from line to line). Note here that since $-M\le \log_2(\gamma)  - \log_2(T)$, by choosing $\gamma$ small enough (once $h$ is fixed), one can always assume that $J(N+1)\ge -M$, which we will do now. 

Then Lemmas \ref{lem-HD} and \ref{lem-asympt} show that for any $-I\le i\le N+1$, on $\mathcal E(A,\delta, I)$, 
$$ \sum_{k\in \widehat \K_i} \sum_{x\in \cR_k} \varphi_T(x-S_k) \bP_x(H^+_{\cR_k}=\infty)\cdot  \1\{x\in Q(S_k,r_{J(i)})\}\le C|\widehat \K_i| \frac{r_{J(i)}^2}{T} \le CL_i\frac{r_{J(i)}^2}{T} \le C2^{-h}\frac{\zeta}{\log n}. $$
On the other hand, for $i$ such that $r_{J(i)} < r_{i-1}$, and $k\in \widehat \K_i$, we use that by definition the time spent on concentric shells around $S_k$ is bounded, up to distance $r_{i-1}$.
This gives for such $i$, using again Lemma \ref{lem-asympt}, 
\begin{align*}
& \sum_{k\in \widehat \K_i} \sum_{x\in \cR_k} \varphi_T(x-S_k) \cdot  \1\{x\in Q(S_k,r_{i-1})\setminus Q(S_k,r_{J(i)})\} \\
& \le C|\widehat \K_i| \sum_{J(i) \le j \le i-1} \frac{\rho_j r_j^d}{T r_j^{d-4}} \le C|\widehat \K_i| \sum_{J(i) \le j \le i-1} \frac{\log n}{T r_j^{d-6}} \\
& \le C|\widehat \K_i| \frac{\log n}{T}. 
\end{align*}
Moreover, by hypothesis on $\zeta$, one has $\frac{n \log n}{T} \le  C\zeta \cdot (\log n)^{-\frac{2}{d-2}}$,  and by \eqref{eq.Ki} it also holds  $\sum_i |\widehat \K_i| = n$. 
Therefore, by fixing now the constant $h$ large enough, we get for all $n$ large enough, 
$$\Sigma_1\le C\left\{(N-M)2^{-h}\frac{\zeta}{\log n} + \frac{n\log n}{T}\right\}  \le \frac{\zeta}{8}.$$ 
Similarly, using Lemmas \ref{lem-HD} and \ref{lem-asympt}, we get by choosing $I$ large enough, 
$$\Sigma_2 \le C \sum_{-M\le i \le -I}  |\widehat \K_i| \le C 2^{-\frac{2I}{d-2}} \cdot \zeta\le \frac{\zeta}{8}.$$

We consider the term $\Sigma_3$.
Note that for any $k$, by definition of $\widehat \K_j$, there are at most $C\rho_j r_j^d$ indices $k'\in \widehat \K_j$, 
such that $S_{k'}\in Q(S_k,r_{j-1})$.
Therefore, using Lemma \ref{lem-asympt}, we get for $n$ large enough, 
\begin{align*}
\Sigma_3 & \le C\sum_{-M\le i\le N+1} |\widehat \K_i| \sum_{i\le j\le N+1} \frac{\rho_j r_j^d}{Tr_{i-1}^{d-4}}
\le C \sum_{-M\le i\le N+1} \zeta\cdot \frac{r_{i-1}^2}{T^{\frac 2{d-2}}}\sum_{i\le j\le N+1} \frac{r_j^2 \log n}{Tr_{i-1}^{d-4}} \\
& \le C\, \frac{\zeta\cdot \log n}{T} \sum_{M\le i\le N+1} \frac{r_{N+1}^2}{r_{i-1}^{d-6}T^{\frac 2{d-2}}} 
\le C\frac{n\log n}{T} \le \frac{\zeta}{8}. 
\end{align*}
Next, using simply Lemma \ref{lem-asympt}, we obtain
(choosing first $I$ large enough, and then $\delta$ small enough) 
\[
\Sigma_4 \le C \sum_{-M\le i\le 0} 
\sum_{-M\le j\le 0} \frac{|\widehat \K_i| \cdot |\widehat \K_j|}{T^{\frac{d-2}{2}}} 
 \le C \zeta  \sum_{-M\le i\le -I} A 2^{\frac{2i}{d-2}} + C\delta\zeta
\sum_{-I\le i \le 0} 2^{\frac{2i}{d-2}} \le \frac{\zeta}{8}.
\]

By Lemma \ref{lem:rear}, one has for some $C>0$,
(choosing first $I$ large enough, and then $\delta$ small enough)
\begin{align*}
\Sigma_5&\le \sum_{-M\le i\le N+1} |\widehat \K_i| \sum_{j\ge \max(i,0)} |\widehat \K_j|^{2/d} \rho_j^{1-\frac 2d} \\
& \le C \sum_{-M\le i\le N+1} |\widehat \K_i| \sum_{j\ge \max(i,0)} 2^{\frac{4j}{d(d-2)} - j(1-\frac 2d)} \\
& \le C  \sum_{-M\le i\le N+1} |\widehat \K_i|  2^{-\max(i,0)(1-\frac 2{d-2})} \\
& \le C\zeta  \left\{\delta \sum_{-I\le i\le I} 2^{-i\frac{d-6}{d-2}} + A \sum_{i\ge I} 2^{-i\frac{d-6}{d-2}} +A\sum_{M\le i\le -I} 2^{\frac{2i}{d-2}} \right\}  \le \frac{\zeta}{8},
\end{align*}
This concludes the proof of the proposition.
\end{proof}


\subsection{Dimension five}\label{sec-d5}
We assume here that $d=5$ and let 
$$T:=\lceil \gamma \cdot (\zeta n)^{1/3}\rceil,$$
with $\gamma$ some constant (chosen similarly as in the previous subsection), and 
$$\overline{\rho}:= \zeta^{5/3} \, n^{-7/3}.$$
Define next $\rho_i$, $r_i$, and $L_i$, for $i\in \Z$, by  
$$\rho_i := 2^i \cdot \overline \rho, \qquad r_i^3\cdot\rho_i = C_0\log n, \qquad \text{and}\qquad L_i:= n \cdot 2^{-\frac{2i}{3}},$$
with $C_0$ as in Theorem \ref{lem-AS}. Let for $i\in \Z$, 
$\widehat \K_i:=\K_n(r_i,\rho_i)\setminus 
\bigcup_{ j>i} \K_n(r_j,\rho_j).  $
Then, let $N$ be the smallest integer, such that $1\le r_N\le 2$, and for $A>0$, $\delta>0$, and $0\le I\le N$, let  
\begin{equation*}
\mathcal E(A,\delta, I):=\left( \bigcap_{-I\le i\le I} \left\{|\widehat \K_i|\le \delta L_i\right\}\right)\cap \left( \bigcap_{I<i\le N}  \left\{|\widehat \K_i|\le A L_i\right\}\right).
\end{equation*}
Our main result here is the following proposition, which implies both the upper bound in Theorem \ref{thm:d5}, as well as Theorem \ref{thm:scen-5} for $d=5$. Since this can be 
done in exactly the same way as in dimension $7$ and higher, we will not repeat the arguments here. 
\begin{proposition}\label{prop.Anxin5} 
For any $A>0$, there exist $\delta>0$ and $I\ge 0$, such that for any $n\ge 2$, and $n^{5/7}\cdot \log n \le \zeta\le n$, 
$$\mathcal E(A,\delta, I) \ \subseteq \ \{\xi_n(T) \le \zeta\}.$$
\end{proposition}
\begin{proof}
Given some $I\ge 0$, let 
$$
\widetilde \K_0:= \{0,\dots,n\}\bs \bigcup_{-I\le i\le N}\K_n(r_i,\rho_i).
$$
Note that for any $I$, 
$$
\xi_n(T)\le 2\Sigma_1 + 2\Sigma_2 +  \Sigma_3 + \Sigma_4,$$
where, 
$$\Sigma_1 :=  \sum_{-I\le i\le N} \sum_{k\in \widehat \K_i} \sum_{k'=0}^n  \varphi_T(S_{k'}-S_k) \cdot  \1\{S_{k'}\in Q(S_k,r_{i+1})\},$$
$$\Sigma_2:=  \sum_{-I\le i \le N} \sum_{k\in \widehat \K_i} \sum_{-I\le j\le i} \sum_{k'\in \widehat \K_j} \varphi_T(S_{k'}-S_k) \cdot  \1\{S_{k'}\notin Q(S_k,r_{j+1})\},$$
$$\Sigma_3: = \sum_{k\in \widetilde \K_0} \sum_{x\in \cR_k}  \varphi_T(x-S_k) \cdot  \1\{x\in Q(S_k,r_{-I})\},$$
$$\Sigma_4:= \sum_{k\in \widetilde \K_0} \sum_{x\in \cR_k}  \varphi_T(x-S_k) \cdot  \1\{x\notin Q(S_k,r_{-I})\},$$

Assume now that $\mathcal E(A,\delta, I)$ holds. Let $J$ be the smallest integer, such that $r_J\le \sqrt{T}$. 
Using Lemma \ref{lem-asympt}, and the bound $\sum_i |\widehat \K_i|\le n$, 
we get for some $C>0$, and $n$ large enough, 
\begin{align*}
\Sigma_1&\le  C \sum_{-I\le i \le N} |\widehat \K_i| \left(\sum_{i\le j\le J} \frac{\rho_jr_j^5}{r_j^3} + \sum_{j>\max(i,J)} \frac{\rho_jr_j^5}{Tr_j}\right) \\
& \le C\log n \sum_{-I\le i \le N} |\widehat \K_i| \left(\sum_{i\le j\le J} \frac{1}{r_j} + \sum_{j>\max(i,J)} \frac{r_j}{T}\right)
 \le C\frac{n\log n}{\sqrt T} \le \frac{\zeta}{8},  
\end{align*}
using also the hypothesis on $\zeta$ for the last inequality. The same argument gives as well $\Sigma_3\le \zeta/4$.

Using in addition Lemma \ref{lem:rear}, we get 
taking first $I$ large enough, and then $\delta$ small enough. 
\begin{align*}
\Sigma_2 &\le  C \sum_{-I\le i \le N} |\widehat \K_i | \sum_{-I\le j\le i} |\widehat \K_j|^{2/5} \rho_j^{3/5} 
\le C \frac{\zeta}{n}\cdot \sum_{-I\le i\le N } |\widehat \K_i | \sum_{-I\le j\le i} 2^{j(-\frac{4}{15}+\frac 35)} \\
& \le C \delta\zeta\sum_{-I\le i\le I} 2^{-i/3}  + C\cdot \zeta\cdot A
\sum_{i\ge I} 2^{-i/3}  \le \frac{\zeta}{8},
\end{align*}
The same argument gives as well $\Sigma_4\le \zeta/4$, 
concluding the proof. 
\end{proof}

\section{Lower Bounds}\label{sec-LB}
We prove here the lower bounds in Theorems \ref{thm:d6} and \ref{thm:d5}. 
In fact in dimension $5$ the result covers a larger range of possible values for $\zeta$.
\begin{proposition} \label{prop.lower5}
Assume $d=5$. 
There exist positive constants $\varepsilon_0$ and $\underline{\kappa}$, such that for any $n\ge 2$, and any $\sqrt n (\log n)^3 \le \zeta \le \epsilon_0 n$, 
one has 
\begin{eqnarray*}
\bP\left(\cc{\RR_n}-\bE[\cc{\RR_n}] \le - \zeta\right)  \ge \exp(- \underline{\kappa} \cdot (\frac{\zeta^2}n)^{1/3}). 
\end{eqnarray*}
\end{proposition}
\begin{proof} 
The proof of \eqref{decomp-2} in \cite{ASS19} reveals that for any finite $A,B\subset \Z^d$, one has also 
\begin{equation}\label{decomp.lower}
\cp(A\cup B) \le \cp(A) + \cp(B) - \chi_0(A,B), 
\end{equation}
with 
$$\chi_0(A,B):=\sum_{x\in A\setminus B}\sum_{y\in B} \bP_x(H_{A\cup B}^+ = \infty) G(y-x) \bP_y(H_B^+ = \infty).$$
Now given $n\ge 1$, set $\ell  = \lfloor \frac{n}{10} \rfloor $, and $m=n-\ell$. 
We apply \eqref{decomp.lower} with $A = \RR_m$ and $B=\RR[m, n]$. 
Fix $\varepsilon_0>0$ (later chosen small enough), and define 
$$E:=\left\{\ccc{\RR_n} \ge -\varepsilon_0 n\right\},$$
where we use the notation $\ccc{\RR_n}$ for the centered capacity. 
Using \eqref{decomp.lower}, Lemma \ref{lem.exp.cap}, and Proposition \ref{cor:upward}, we deduce that 
for some constant $c>0$,  
\begin{equation}\label{lower5.1}
\begin{split}
\bP\left(-\varepsilon_0n\le \ccc{\RR_n}\le - \zeta\right) 
 & \ge \ \bP(E, \, \chi_0(\RR_m,\RR[m,n]) \ge 4\zeta) - \bP(\ccc{\RR_m}\ge \zeta) \\
 &\qquad  - \bP(\ccc{\RR[m,n]} \ge \zeta) \\
& \ge\  \bP(E, \, \chi_0(\RR_m,\RR[m,n])\ge 4\zeta) - 2\exp(-c\frac{\zeta^2}{n(\log n)^3}).
\end{split}
\end{equation}
Note that when $\zeta\ge \sqrt n (\log n)^3$, then $\zeta^2/(n(\log n)^3) \ge (\log n) (\zeta^2/n)^{1/3}$, and 
therefore the last term above is negligible. 
Now, let $\rho>0$ be some small constant (to be fixed later) and consider the event 
$$F:= \{\|S_k\| \le \rho\cdot n^{2/3}\, \zeta^{-1/3}, \quad \text{for all } k\le n\}.$$
Note that by \eqref{Green} and \eqref{cap.def2}, on the event $F$, 
\begin{equation}\label{chi0Rm}
\chi_0(\RR_m,\RR[m,n])\ \ge\ c_\rho\cdot  \frac{\zeta}{n^2}\cdot  \cp(\RR[m,n]) \cdot \left(\cp(\RR_n) - \cp(\RR[m,n])\right),
\end{equation}
for some constant $c_\rho>0$, going to infinity as $\rho$ goes to zero. 
Furthermore, by \eqref{cap.subadd}, one has 
$$\cp(\RR_n) \le \cp(\RR_m) + \cp(\RR[m,n]),$$
and thus by Lemma \ref{lem.exp.cap} and Proposition \ref{cor:upward}, by taking $\varepsilon_0$ small enough, we get for $n$ large enough, 
\begin{align*}
\bP\left(\cp(\RR[m,n]) \le \gamma_5 \frac{\ell}{2},\, E\right) \ &  
\le\ \bP\left( \cp(\RR_m) \ge \gamma_5 (m +\ell/3) \right)\\
&\le \ 
\bP\left(  \ccc{\RR_m} \ge \gamma_5 \frac{\ell}{10} \right)\le \ \exp\left(- c' \frac{n}{(\log n)^3}\right),
\end{align*}
for some constant $c'>0$, and with $\gamma_5$ as in \eqref{cap.limit}. 
Similarly one has for some possibly smaller constant $c'>0$,
$$
\bP\left(\cp(\RR_n) - \cp(\RR[m,n])\le \gamma_5 \frac n 4,\, E\right)\ 
\le\ \bP\left(\ccc{\RR[m,n]}\ge \gamma_5\ell \right) \le \ \exp\left(- c' \frac{n}{(\log n)^3}\right).
$$
Then \eqref{chi0Rm} gives
\begin{align*}
\bP( \chi_0(\RR_m,\RR[m,n]) \ge \frac{c_\rho \gamma_5^2}{100}\cdot \zeta, \, E)\ 
& \ge\ \bP( \chi_0(\RR_m,\RR[m,n]) \ge \frac{c_\rho \gamma_5^2}{100}\cdot \zeta, \, E \cap F) \\
& \ge \ \bP(E\cap F) - 2 \exp\left(- c' \frac{n}{(\log n)^3}\right)\\
&\ge \ \bP(F) - \bP(E^c) - 2\exp\left(- c' \frac{n}{(\log n)^3}\right).
\end{align*}
Coming back to \eqref{lower5.1}, and choosing $\rho$, such that $c_\rho \ge 300/\gamma_5^2$, we deduce that 
\begin{align*}
\bP\left( \ccc{\RR_n} \le - \zeta\right)  \ & = \ \bP\left(-\varepsilon_0n
\le \ccc{\RR_n}\le - \zeta\right)+ \bP(E^c)\\
& \ge \ \bP(F) - 2\exp\left(- c' \frac{n}{(\log n)^3}\right) - 2\exp(-c\, (\log n)\cdot \zeta^{2/3} n^{-1/3}).   
\end{align*}
Moreover, it is well known that for any $\rho>0$, there exists $\kappa>0$, such that 
$$\bP(F) \ \ge\  \exp(-\kappa \cdot \zeta^{2/3}n^{-1/3}),$$
and this concludes the proof. 
\end{proof}

In dimension $6$ and more the result reads as follows. 
\begin{proposition}\label{prop.lower6}
Assume $d\ge 7$. There exist positive constants $\varepsilon_0$, $K$ and $\underline{\kappa}$, 
such that for any $n\ge 2$ and any $K n^{\frac{d-2}{d}}  \le \zeta \le \varepsilon_0\, n$, one has 
$$
\bP\left(\cc{\RR_n} - \bE[\cc{\RR_n}]\le -\zeta\right) 
\ \ge\  \exp\left(- \underline{\kappa} \cdot \zeta^{1-\frac{2}{d-2}} \right).$$
In dimension $d=6$, the same result holds for $n^{\frac{d-2}{d}}(\log \log n)^2 \le \zeta \le \varepsilon_0\, n$. 
\end{proposition}
\begin{proof}
We prove the result for $d\ge 7$ to keep notation simple, but the same argument works as well for $d=6$. 
Set $\ell : = \lfloor 5\zeta/\gamma_d \rfloor$. 
Using \eqref{cap.subadd}, Lemma \ref{lem.exp.cap}, and Proposition \ref{cor:upward},  
we obtain that for some constant $c>0$,
\begin{align*}
\bP\left(\ccc{\RR_n}\le - \zeta\right) & \ge \bP\left(\ccc{\RR_\ell} \le -3\zeta\right) 
- \bP\left(\ccc{\RR[\ell,n]} \ge \zeta\right)\\
& \ge \bP\left(\cc{\RR_\ell}\le \zeta\right)  - \exp(-c \cdot \frac{\zeta^2}{n}),
\end{align*}
at least provided $\zeta$ is large enough, which one can always assume. 
Now the hypothesis on $\zeta$ implies that the last term is negligible, provided $K$ is chosen large enough, and 
by the same argument as in the proof of Proposition \ref{prop.lower5}, 
one can see that the first term on the right-hand side is of the right order (which is of the order of the event $F$ where the walk stays confined in a ball of radius $c'\zeta^{1/(d-2)}$, with $c'>0$ small enough,  during the whole time $\ell$). This concludes the proof of the proposition. 
\end{proof}


\section{The Gaussian regime} \label{sec.Gaussian}
The starting point to proving Theorem \ref{theo.Gaussian} is a
standard dyadic decomposition which follows from using \eqref{decomp-1}
repeatedly along a dyadic scheme. For any $L\ge 1$, and $n\ge 2^L$,
\begin{equation}\label{UBd5-1}
\cp(\cR_n)-\bE[\cp(\cR_n)]=  \sum_{i=1}^{2^L} \left( \cp(\cR_i^L) - \bE[ \cp(\cR_i^L)] \right)
-\sum_{\ell=1}^{L}\sum_{i=1}^{2^{\ell-1}} Y_i^\ell,
\end{equation}
where $Y_i^\ell:= \chi_\C(\cR_{2i-1}^\ell, \cR_{2i}^\ell ) -
\bE[\chi_\C(\cR_{2i-1}^\ell, \cR_{2i}^\ell )]$,  
and the $\{\cR_i^\ell\}_{i=1,\dots,2^\ell}$, are independent
ranges of length $n2^{-\ell}$ (the time-length is not
exactly equal for each of them since we do not suppose that $n$ 
is of the form $n=2^K$, for some $K\ge 1$, but they
differ by at most one unit).

A gaussian-type fluctuation is due to the sum of the $2^L$ self-similar
terms in \reff{UBd5-1}, after $L$ is chosen appropriatly. It is classical
(see \cite{Chen})
to use G\"artner-Ellis' Theorem after we show that the contribution
of the $Y_i^\ell$ is negligible. 
Thus, the main technical novelty of this section is the
stretched exponential moment bound \reff{stretch-chi}, which is performed
in Section \ref{subsec.2.gauss}. 

After recalling some well-known results
in Section \ref{subsec.1.gauss},
we conclude the proof of Theorem 
\ref{theo.Gaussian} in Section \ref{subsec.3.gauss}.

\subsection{Preliminary results}\label{subsec.1.gauss}
We first state an instance 
of G\"artner-Ellis' Theorem (see Theorem 2.3.6 in \cite{DZ}). 
\begin{theorem}[G\"artner-Ellis]\label{theo-GE}
Let $\{X_n\}_{n\ge 0}$ be a sequence of real random variables.
Let $\{b_n\}_{n\ge 0}$ going to infinity, and
for any $\theta\in \R$, 
\begin{equation*}
\text{If $\forall\theta\in \R$}\quad
\lim_{n\to\infty} \frac{1}{b_n} \log \bE[\exp(\theta b_n\cdot X_n)]=
\frac{\sigma^2}{2}\cdot \theta^2,\quad 
\text{then, $\forall\lambda>0$,}\quad
\lim_{n\to\infty} \frac{1}{b_n}
\log \bP(X_n> \lambda)=-\frac{\lambda^2}{2\sigma^2}.
\end{equation*}
\end{theorem}
We recall now a large deviation estimates  for variables 
with stretched exponential moment.  
\begin{theorem}[A. Nagaev \cite{anagaev}]\label{theo-nagaev}
Let $\{Y_n\}_{n\ge 0}$ be a sequence of centered random variables, 
such that 
$\bE[\exp(\kappa |Y_1|^\alpha)]<\infty$, for some constants $\kappa>0$, and $\alpha\in (0,1]$. Then 
there are positive constants $c$ and $C$, such that for any $n\ge 1$ and any $t>n^{\frac 1{2-\alpha}}$, 
\begin{equation*}
\bP\big(Y_1+\dots+Y_n>t\big)\le C\exp(-c t^\alpha). 
\end{equation*}
\end{theorem}
\subsection{Stretched exponential moment of the cross term}\label{subsec.2.gauss}
The heart of the proof of Theorem \ref{theo.Gaussian} 
use Theorem \ref{theo.stretched.exp} below which is more general
than  \reff{stretch-chi}, and has interest of its own.
It is analogous to the arguments of \cite{AS20c}. 

Define for any subsets $A,B\subseteq \Z^d$, 
$$\Gamma(A,B) = \sum_{x\in A} \sum_{y\in B} G(y-x) \bP_y(H_B^+ =\infty). $$ 
Recall that $0\le \chi_\C(A,B) \le 2 \Gamma(A,B)$, for any $A,B\subseteq \Z^d$.  

\begin{theorem}\label{theo.stretched.exp}
Let $\cR_\infty$ and $\widetilde \cR_\infty$ be the ranges of two independent random walks on $\Z^d$, with $d\ge 7$. There exist positive constants $c_1,c_2$, such that 
for all $t$ large enough, 
$$\exp(-c_1t^{1-\frac 2{d-2}} ) \le \bP(\Gamma(\widetilde \cR_\infty, \cR_\infty ) > t) \le \exp(- c_2 t^{1-\frac 2{d-2}}). $$ 
\end{theorem}
Let us notice that in the definition of $\Gamma$ it is fundamental to keep the escape probabilities, in other words one cannot simply bound them by one. Indeed one could show that the tail distribution of  
$\Gamma'(\cR_\infty,\widetilde \cR_\infty):= \sum_{x\in \cR_\infty} \sum_{y\in \widetilde \cR_\infty} G(x-y)$ obeys a different decay at infinity.  

\begin{proof}[Proof of Theorem \ref{theo.stretched.exp}] 
We start with the lower bound. Observe that $\Gamma(\cdot,\cdot)$ is increasing in both arguments for the inclusion of sets, thus for any $n\ge 1$, 
$$\Gamma(\widetilde \cR_\infty, \cR_\infty ) \ge \Gamma (\widetilde \cR_n, \cR_n).$$
Therefore the lower bound is obtained by forcing the two walks to stay confined in a ball of radius $t^{\frac{1}{d-2}}$ for a time $Ct$, with $C>0$ large enough, exactly as in the proof of Proposition \ref{prop.lower5}.

We now move to the upper bound. The proof is obtained in three steps. 
In the first step, we reduce the time window of one walk to a finite interval, as follows. Observe that for any integer $n\ge 1$, 
\begin{align*}
& \bE[\Gamma(\widetilde \cR_\infty, \cR[n,\infty))]  \le \bE\left[\sum_{k=0}^\infty \sum_{\ell = n}^\infty G(S_k-\widetilde S_\ell)\right]= \sum_{k=0}^\infty \sum_{\ell = n}^\infty\bE[G(S_{k+\ell})] \\
& = \sum_{k=n}^\infty (k+1-n)\bE[G(S_k)]\le C\sum_{k=n}^\infty\frac{k+1-n}{k^{\frac{d-2}{2}}} \le \frac{C}{n^{\frac{d-6}{2}}},  
\end{align*}
for some constant $C>0$. Therefore if we let $n :=\exp(t^{1-\frac{2}{d-2}})$, then by Markov's inequality, 
$$\bP(\Gamma(\widetilde \cR_\infty, \cR[n,\infty))\ge 1) \le \bE[\Gamma(\widetilde \cR_\infty, \cR[n,\infty))] \le C\exp(- \frac{d-6}{2}\cdot t^{1-\frac{2}{d-2}}), $$ 
and thus, due to the inequality 
$$\Gamma(\widetilde \cR_\infty, \cR_\infty) \le \Gamma (\widetilde \cR_\infty, \cR_n) + \Gamma (\widetilde \cR_\infty, \cR[n,\infty)), $$
it just remains to bound the first term on the right-hand side.

In a second step we claim that for any subset $\Lambda\subseteq \Z^d$, and any $t\ge 1$, 
\begin{equation}\label{Claim.Gamma}
\bP(\Gamma(\widetilde \cR_\infty, \Lambda) > t ) \le  \exp\left(-\frac{t\cdot \log 2}{2\sup_{x\in \Z^d} \bE_x[\Gamma(\widetilde \cR_\infty, \Lambda)]}\right). 
\end{equation} 

To see this, we use again that for any $A,B\subseteq \Z^d$, one has 
$\Gamma(A\cup B,\Lambda) \le \Gamma(A,\Lambda) + \Gamma(B,\Lambda)$. Thus the Markov property and Markov's inequality show that the random variable 
$\frac{\Gamma(\widetilde \cR_\infty, \Lambda)}{2\sup_{x\in \Z^d} \bE_x[\Gamma(\widetilde \cR_\infty, \Lambda)]}$ is stochastically bounded by a Geometric random variable with mean $2$, from which \eqref{Claim.Gamma} follows immediately. Note also that for any $x$, 
$$ \bE_x[\Gamma(\widetilde \cR_\infty, \Lambda)] \le  \sum_{z\in \Lambda} G\star G(z-x)\cdot  \bP_z(H_{\Lambda}^+ = \infty) =:\mathcal F(\Lambda - x),$$
where we recall that $G\star G$ is the convolution of $G$ with itself, and 
$$\mathcal F(\Lambda) := \sum_{z\in \Lambda} G\star G(z) \cdot \bP_z(H_\Lambda^+=\infty). $$ 
Thus it amounts to show that for some positive constants $c$ and $C$, one has 
\begin{equation}\label{Claim.2.F}
\bP\left(\sup_{x\in \Z^d} \mathcal F(\cR_n-x) > Ct^{\frac 2{d-2}}\right) \le C\exp(-ct^{1-\frac 2{d-2}}), \quad \text{with }n=\exp(t^{1-\frac{2}{d-2}}), 
\end{equation} 
which is our third and last step. Note that $\mathcal F$ is also subadditive in the sense that for any $A,B\subseteq \Z^d$, 
$\mathcal F(A\cup B) \le \mathcal F(A) + \mathcal F(B)$. This allows to partition the range into different pieces, according to the occupation density 
in a certain neighborhood, and then bound $\mathcal F$ on each of them. 
To be more precise, set $\rho_0 := t^{-\frac 2{d-2}}$, and then for $i\ge 0$, define $\rho_i$ and $r_i$ by  
$$\rho_i := 2^{-i}\rho_0,\quad \text{and} \quad \rho_i r_i^{d-2}= C_0 \log n,$$
with $C_0$ as in \eqref{hyp.r}. 
Then let $\cR_n(r_i,\rho_i):= \{S_k,\ k\in \K_n(r_i,\rho_i)\}$, and 
$$
\Lambda_i := \cR_n(r_i,\rho_i) \bs 
\bigcup_{0\le j< i} \cR_n(r_j,\rho_j) ,\qquad  \Lambda_i^*:
= \cR_n \bs  \bigcup_{0\le j<  i}\Lambda_i. 
$$  
By Theorem \ref{lem-AS}, one has for any $i\ge 0$,  
$$\bP(|\Lambda_i|\ge 2^{\frac{2i}{d-2}} t ) \le C\exp(-\kappa t^{1-\frac 2{d-2}}),$$
for some positive constants $C$ and $\kappa$, and in fact for $i>\frac{d-2}{2} \log_2(n+1)$, the above probability is zero, since by definition $|\Lambda_i|\le n+1$. 
Therefore, if we let 
$$\mathcal E := \left\{|\Lambda_i|\le 2^{\frac{2i}{d-2}} t, \ \text{for all }i\ge 0\right\},$$
then the above discussion shows that 
$$\bP(\mathcal E^c) \le C\exp(-(\kappa/2)\cdot t^{1-\frac 2{d-2}}),$$
at least for $t$ large enough. We now show that for some constant $C>0$, 
\begin{equation}\label{claim.E.F}
\mathcal E \subseteq \left\{\sup_{x\in \Z^d} \mathcal F(\cR_n-x) \le Ct^{\frac 2{d-2}}\right\},
\end{equation}
which will conclude the proof of the theorem. To simplify notation we only bound $\mathcal F(\cR_n-x)$ for $x=0$, but it should be clear from the proof that all 
our estimates are uniform with respect to $x$. We partition space into shells $(\mathcal S_k)_{k\ge 0}$, defined by $\mathcal S_0 := Q(0,r_0)$, and $\mathcal S_k:= Q(0,r_k)\bs Q(0,r_{k-1})$ for $k\ge 1$. By subadditivity, one has 
$$\mathcal F(\cR_n) \le \sum_{k\ge 0} \mathcal F(\mathcal S_k \cap \cR_n). $$
The proof of Lemma \ref{lem-asympt} shows that $G\star G(z) \le C \|z\|^{4- d}$, and thus Lemma \ref{lem-HD} gives 
$$\mathcal F(\mathcal S_0\cap \cR_n) \le  \mathcal F(\mathcal S_0) \le Cr_0^2 \le Ct^{\frac{2}{d-2}}. $$ 
Then for $k\ge 1$, we write 
$$\mathcal F(\mathcal S_k\cap \cR_n) \le \sum_{i=0}^k \mathcal F(\mathcal S_k\cap \Lambda_i) + \mathcal F(\mathcal S_k\cap \Lambda_{k+1}^*).$$
On one hand one has on the event $\mathcal E$, 
$$\mathcal F(\Lambda_0 \cap \mathcal S_0^c) \le C\frac{|\Lambda_0|}{r_0^{d-4}} \le Ct^{\frac 2{d-2}}. $$
On the other hand, for any $i\ge 1$, 
\begin{align*}
 \sum_{k\ge  i } \mathcal F(\mathcal S_k\cap \Lambda_i)  \le \sum_{z\in \Lambda_i \cap Q(0,r_{i-1})^c} G\star G(z)\le \sum_{z\in \Lambda_i \cap Q(0,r_{i-1})^c} \frac{C}{1+\|z\|^{d-4}}   \le C\rho_i^{1-\frac{4}{d}} |\Lambda_i|^{4/d}, 
\end{align*}
using the same argument as in the proof of Lemma \ref{lem:rear} for the last inequality.  Thus on the event $\mathcal E$, we get 
$$\sum_{k\ge  i } \mathcal F(\mathcal S_k\cap \Lambda_i)  \le C 2^{-i\frac{d-6}{d-2}} t^{\frac 2{d-2}}. $$ 
It follows that on $\mathcal E$, 
$$\sum_{i\ge 1} \sum_{k\ge i}  \mathcal F(\mathcal S_k\cap \Lambda_i)  \le C 2^{-i\frac{d-6}{d-2}} t^{\frac 2{d-2}} \le Ct^{\frac 2{d-2}} . $$ 
Similarly, one has 
$$\sum_{k\ge 1}  \mathcal F(\mathcal S_k\cap \Lambda_{k+1}^*) \le C\sum_{k\ge 1} \frac{\rho_k r_k^d}{r_{k-1}^{d-4}} \le C\frac{\log n}{r_0^{d-6}} \le Ct^{\frac 2{d-2}}. $$  
Altogether this proves \eqref{claim.E.F}, and concludes the proof of the theorem. 
\end{proof}

\subsection{Proof of Theorem \ref{theo.Gaussian}}\label{subsec.3.gauss} 

Let $\{\zeta_n\}_{n\ge 0}$ be a sequence as in the statement of Theorem \ref{theo.Gaussian}, and let $L$ be the integer such that 
$2^{L-1} \le \zeta_n  <2^L$. 

We first show that the cross terms appearing in 
\eqref{UBd5-1} are negligible. 
Applying Theorems \ref{theo-nagaev} and \ref{theo.stretched.exp}, we get that for any $\delta>0$, and any $\ell \le L$, 
$$\limsup_{n\to \infty} \frac{n}{\zeta_n^2}\cdot \log \bP\left(\pm \sum_{i=1}^{2^\ell} Y_i^\ell \ge \frac{\delta\zeta_n}{L} \right)  = - \infty. $$ 
By using a union bound we also deduce 
$$\limsup_{n\to \infty} \frac{n}{\zeta_n^2}\cdot \log \bP\left(\pm \sum_{\ell = 1}^L \sum_{i=1}^{2^\ell} Y_i^\ell \ge \delta\zeta_n \right)  = - \infty. $$ 
Thus indeed the cross terms in \eqref{UBd5-1} can be ignored, and we focus now on proving the Moderate Deviation Principle for the first sum.

For simplicity, let $Z_i:= |\cR_i^L| -\bE[ |\cR_i^L|]$. We apply Theorem \ref{theo-GE} with $X_n:= \frac{\pm 1}{\zeta_n}\sum_{i=1}^{2^L}Z_i$, and 
$b_n:=\zeta_n^2/n$. One has using independence, and the fact that $\frac{\zeta_n}{n}\cdot |Z_1|$ is bounded, 
\begin{equation*}
\bE[\exp(\theta b_n X_n]=
\Big(\bE[\exp(\theta \frac{\zeta_n}{n} Z_1] \Big)^{2^L}\\
= \Big(1+\frac{\theta^2}{2}\big(\frac{\zeta_n}{n}\big)^2\cdot
\bE[Z_1^2]+
\mathcal O\big(\big(\frac{\zeta_n}{n}\big)^3\cdot \bE[|Z_1|^3]\big) \Big)^{2^L}.
\end{equation*}
Note that $2^L\cdot \bE[Z_1^2]/n$ converges to $\sigma^2>0$, 
and that the fourth centered moment of $\cp(\cR_n)$ 
is $\mathcal O(n^2 (\log n)^2)$. This can be seen as for the volume of the range, following the same proof as in \cite{LG86}.
Thus, using that $\bE[|Z_1|^3]\le \bE[Z_1^4]^{3/4}$, we have
\[
\big(\frac{\zeta_n}{n}\big)^3\cdot \bE[|Z_1|^3]\le C
\big(\frac{\zeta_n\log n}{n}\big)^{3/2}.
\]
It follows that for any $\theta\in \R$,
\begin{equation*}
\lim_{n\to\infty} \frac{n}{\zeta_n^2} \log \bE[
\exp\Big(\theta\frac{\zeta_n}{n} X_n\Big)=
\frac{\sigma^2}{2} \theta^2, 
\end{equation*}
and one can then apply G\"artner--Ellis' Theorem, which concludes the proof of Theorem \ref{theo.Gaussian}. 


\section{Upward Deviations}\label{sec-HK}
We prove here Theorem~\ref{theo:upward}. 
Thanks to our decomposition \reff{decomp-1}, we can
adapt the approach of Hamana and Kesten \cite{HK}, 
who proved a similar result for the size of the range.

The approach of Hamana and Kesten is based on first proving an approximate subadditivity relation for the probability of upward deviations, that is the existence of some constants 
$\chi \in (0,1)$, $c>0$, and $C>0$, such that for any $m,n\ge 1$ integers, and
$y,z$ positive reals, 
\begin{equation}\label{KH-main}
\bP\big(|\cR_{m+n}|\ge y+z - C a(m,n) \big)\ \ge\ 
c\, \chi^{a(m,n)}\, \bP\big(|\cR_n|\ge y\big)\bP\big(|\cR_m|\ge z\big),
\end{equation}
with
\begin{equation*}
a(m,n):=(n\cdot m)^{\frac{1}{d+1}}. 
\end{equation*}
The second step, which is general 
and only based on \eqref{KH-main} and the fact that (when $d\ge 2$) one has $\lim_{m,n\to\infty} \frac{a(m,n)}{n\vee m}=0$, 
shows that the following limit exists, 
\[
\psi(x):=- \lim_{n\to\infty} \frac{1}{n}
\log \bP\big(|\cR_n|\ge x\cdot n \big), \quad \text{for all }x>0,
\]
and that $\psi$ is continuous and convex on $[0,1]$. 
Here we prove an analogous result as \eqref{KH-main}, and use their general argument to conclude.

\begin{proof}[Proof of Theorem \ref{theo:upward}]
We first prove an analogous result as \eqref{KH-main}, but with $a(m,n)$ replaced by the function:  
\[
\widetilde a(m,n)=(n\cdot m)^{\frac{1}{d-1}}.
\]
In other words we establish the following inequality. There exists 
$\chi \in (0,1)$, and $C>0$, such that for any $m,n$ integers and $y,z$ positive reals,
\begin{equation}\label{amin-1}
\bP\left(\cc{\cR_{m+n}}\ge y+z - C\, \widetilde  a(m,n) \right)\ge
\frac 12 \chi^{\widetilde a(m,n)} \bP\left(\cc{\cR_n}\ge y\right)
\bP\left(\cc{\cR_m}\ge z\right).
\end{equation}
The first step is to obtain the analogue of the following simple
deterministic bound used in \cite{HK}: if $\cR_n$ and $\widetilde \cR_m$ are two independent 
copies of the range, there is a positive constant $C$, such that
for any $r\ge 1$
\[
\frac{1}{|Q(r)|}\sum_{z\in Q(r)}
|(z+\cR_n)\cap \widetilde \cR_m|
\le C\, \frac{n\cdot m}{r^d}.
\]
The corresponding bound in our context reads as follows: 
\begin{equation}\label{amin-2}
\frac{1}{|Q(r)|}\sum_{z\in Q(r)} \sum_{x\in \cR_n}
\sum_{y\in \widetilde \cR_m} G(x-y+z)\le C \, \frac{n\cdot m}
{r^{d-2}},
\end{equation}
and is a direct consequence of \eqref{Green} 
and the fact that for any $x\in \Z^d$, and for some constant $C>0$,
$$
\sum_{z\in Q(r)} \frac{1}{1+\|z-x\|^{d-2}} \le C\, r^2.
$$
Now to lighten notation, we simply write $a=\widetilde a(m,n)=\lfloor (mn)^{\frac{1}{d-1}}\rfloor$. Using that the capacity is translation-invariant, we deduce
\begin{equation}\label{amin-3}
\begin{split}
\cc{\cR_{m+n+a}}&\stackrel{\eqref{cap.mon}}{\ge}\cc{\cR_n\cup \cR[n+a,n+m+a]}\\
&\stackrel{\eqref{decomp-1}}{=} \cc{\overline \cR_n} +\cc{\widetilde \cR_m} -\chi_C(\overline \cR_n,\widetilde \cR_m+S'_a),
\end{split}
\end{equation}
with $\overline \cR_n:= \cR_n - S_n$, $S'_a := S_{n+a}- S_n$, and $\widetilde \cR_m:=\cR[n+a,n+m+a]-S_{n+a}$. 
The Markov property implies that
$\overline \cR_n$ and $\widetilde R_m$ are independent, and distributed as $\cR_n$ and $\cR_m$ respectively. 
Furthermore, 
\begin{equation}\label{amin-4}
\chi_C(\overline \cR_n,\widetilde \cR_m+ S'_a)\stackrel{\eqref{decomp-2}}{\le} \sum_{x\in \overline \cR_n} \sum_{y\in \widetilde \cR_m} G(x-y- S'_a).
\end{equation}
Now, one idea of Hamana and Kesten \cite{HK} is to bound the law of $S_a'$ by a uniform
law on the cube $Q(a/d)$. Indeed for any $x\in Q(a/d)$, for which $\bP(S_a=x)\neq 0$, one has 
\begin{equation}\label{inf.hk}
\bP(S_a'=x)\ge \frac{1}{(2d)^a},
\end{equation}
since there is at least one path of length $a$ going from $0$ to $x$. 
Write $\overline Q(a/d)$ for the set of sites $x\in Q(a/d)$, for which $\bP(S_a=x)\neq 0$. Then for any $x\in \overline Q(a/d)$, and any $\alpha>0$, 
\begin{equation*}
\bP \left( \cp(\cR_{m+n+a})\ge  z+y-\frac{\alpha}{2} \right) \stackrel{\eqref{amin-3}}{\ge} 
\bP(S_a'=x)\cdot \bP\left(\cp(\overline \cR_n)\ge z,\cp(\widetilde \cR_m)\ge y,
\chi_C(\overline \cR_n,\widetilde \cR_m+x)\le \frac{\alpha}{2}\right).
\end{equation*}
Integrating with respect to the uniform measure on $\overline Q(a/d)$, we get
\begin{equation}\label{KH-2}
\begin{split}
\bP\big( \cp(\cR_{m+n+a}) & \ge z+y-\frac{\alpha}{2}\big) \ \stackrel{\eqref{inf.hk}}{\ge} \ 
\frac{1}{(2d)^a} \\
&  \times \frac{1}{|\overline Q(a/d)|}\sum_{x\in \overline  Q(a/d)} 
\bP\left(\cp(\overline \cR_n)\ge z,\cp(\widetilde \cR_m)\ge y,
\chi_C(\overline \cR_n,\widetilde \cR_m+x)\le \frac{\alpha}{2}\right).
\end{split}
\end{equation}
We need now to estimate the mean of $\chi_C(\overline \cR_n,\widetilde \cR_m+\cdot)$
with respect to the uniform measure. According to \eqref{amin-2}, there
is a positive constant $C$, such that
\begin{equation}\label{amin-5}
\frac{1}{|\overline Q(a/d)|}\sum_{x\in \overline Q(a/d)} \chi_C(\overline \cR_n,\widetilde \cR_m+x)\le  C \frac{m\cdot n}{a^{d-2}}\le  Ca,
\end{equation}
where the last inequality follows from the definition of $a$. 
Then by Chebychev's inequality, we obtain 
\begin{equation}\label{cheb}
\frac{1}{|\overline Q(a/d)|}\sum_{x\in \overline Q(a/d)} \1(\chi_C(\overline \cR_n,\widetilde \cR_m+x)\le
2Ca)\ge \frac{1}{2}.
\end{equation}
As a consequence, 
\begin{equation*}
\begin{split}
& \ \bP\big(\cp(\cR_{m+n}) \ge z+y-a- 4Ca\big) \ \stackrel{\eqref{cap.subadd}, \eqref{cap.card}}{\ge} \  
\bP\big(\cp(\cR_{m+n+a})\ge z+y- 4Ca\big)\\
& \stackrel{\eqref{KH-2}}{\ge}\  \frac{1}{(2d)^a} \cdot \bE\Big[\1(\cp(\overline \cR_n)\ge z)\cdot 
\1(\cp(\widetilde \cR_m)\ge y)
 \times \frac{1}{|\overline Q(a/d)|}\sum_{x\in \overline Q(a/d)} \1(\chi_C(\overline \cR_n,\widetilde \cR_m+x)\le 2 Ca)\Big]\\
& \stackrel{\eqref{cheb}}{\ge}\  \frac{1}{2(2d)^a} \cdot \bP\big(\cc{\cR_n}\ge z\big)
\bP\big(\cc{ \cR_m}\ge y\big),
\end{split}
\end{equation*}
proving \eqref{amin-1}, with $\chi = 1/(2d)$.

It then follows from the general arguments of Hamana and Kesten, see Lemma 3 in \cite{HK}, that the following limit exists for all $x>0$: 
$$\psi_d(x)\ := \ -\lim_{n\to \infty}\  \frac 1n \log \bP\left(\cp(\cR_n) \ge nx\right).$$

We now prove that the range for which $\psi_d(x)$ is finite is not empty.
Define for $n\ge 0$, 
\begin{equation}\label{cap.cn}
c_n:=\max_{\gamma:\{0,\dots,n\}\to \Z^d} \cp(\{\gamma(0),\dots,\gamma(n)\}),
\end{equation}
where the max is taken over all nearest neighbor paths of length $n+1$. 
By \eqref{cap.subadd}, it follows that $c_{n+m}\le c_n + c_m$, for all $n,m\ge 0$, so that by Fekete's lemma, the limit $\lim_{n\to \infty} c_n/n$ exists. Call $\gamma_d^*$ this limit. 
Note that $\psi_d(x)$ is finite on $[\gamma_d,\gamma_d^*]$, since the probability that the simple random walk follows the path realizing the maximum in \eqref{cap.cn} is larger than or equal to $1/(2d)^{n+1}$, so that $\psi_d(x)\le \log(2d)$, for all $x\le \gamma_d^*$. Conversely, by definition of $c_n$, one has $\psi_d(x)=\infty$ for all $x> \gamma_d^*$. 
Furthermore, it follows from Lemma 3 and Proposition 4 in \cite{HK}, that $\psi_d$ is continuous, and convex on $(0,\gamma_d^*]$. 
Now Proposition \ref{cor:upward} and Lemma \ref{lem.exp.cap} show that when $d=5$, $\psi_d(x) \ge c(x-\gamma_5)^3$, for all $x\ge \gamma_d$. 
Likewise, when $d\ge 6$, we get $\psi_d(x) \ge c(x-\gamma_d)^3$, for $\gamma_d \le x\le 1$. 
Using convexity, this also shows that $\psi_d$ is increasing on $[\gamma_d,\gamma_d^*]$. 
In addition one has $\psi_d(x) = 0$ for all $x<\gamma_d$, by definition of $\gamma_d$ as the limit of the (normalized) expected capacity, and using that by \eqref{cap.card}, $\cp(\cR_n)\le n$.

Finally we show that $\gamma_d^*>\gamma_d$.

Consider $\D_n$ the set of {\it no double backtrack at even times} 
paths of length $n+1$ that we introduced in \cite{AS2}. 
By definition, this is simply the set of paths $\gamma:\{0,\dots,n\}\to \Z^d$, such that for any even $k\le n$, one has $\gamma(k+2)\neq \gamma(k)$. The only important property we need is that from a
no-backtrack walk $\widetilde S$, and a sum of independent geometric
variables $\{\xi_i,\ i\in \N\}$, with parameter $1/(2d)^2$, we can build a simple random walk $S$
such that
\begin{equation*}
\cR[0,n+2\sum_{i\le n/2} \xi_i]=\widetilde \cR_n. 
\end{equation*}
Thus, for any $\alpha>0$, we have by \eqref{cap.subadd} and \eqref{cap.card},
\begin{equation*}
\cp(\widetilde \cR_n) 
\ge \cp(\cR_{(1+\alpha)n})-
\1\left(\sum_{i\le n/2} \xi_i< \frac{\alpha n}{2}\right) \cdot (1+\alpha)(n+1).
\end{equation*}
By taking the maximum over $\D_n$ on the left hand side, and
then the expectation on the right hand side, we obtain
\begin{equation}\label{upw-const4}
c_n\ge \max_{\pi\in \D_n} \cp(\pi) \ge \bE[ \cp(\cR_{(1+\alpha)n})]-
(1+\alpha)(n+1)\cdot \bP\left(\sum_{i\le n/2} \xi_i<\frac{\alpha n}{2}\right).
\end{equation}
Now take $\alpha< 1/(2d)^2$, and use Chebyshev's inequality, to see that  the last term of \reff{upw-const4} is $\mathcal O(1)$. 
Together with Lemma \ref{lem.exp.cap} it implies that 
\begin{equation*}
c_n\ge \gamma_d(1+\alpha)n-\cO(\sqrt n),
\end{equation*}
which indeed proves that $\gamma_d<\gamma_d^*$. 
\end{proof}

\vspace{0.3cm}
\noindent{{\bf Acknowledgements:} Perla Sousi participated at an early stage
of the project and we thank her for stimulating discussions, and her 
proof of Lemma~\ref{lem-HD}. 
We also thank Quentin Berger and Julien Poisat for the idea of 
considering the polymer melt. Finally, we thank two anonymous referees,
whose suggestions were crucial in clarifying the arguments.
The authors were partly supported by the French Agence Nationale 
de la Recherche under grants ANR-17-CE40-0032 and ANR-16-CE93-0003. }


\begin{thebibliography}{99}




\bibitem[AS17a]{AS} Asselah Amine; Schapira Bruno. Moderate deviations for the range of a transient random walk: path concentration. Ann. Sci. \'Ec. Norm. Sup\'er. (4) 50 (2017), 
755--786.
 
\bibitem[AS17b]{AS2} Asselah Amine; Schapira Bruno. 
Boundary of the range of transient random walk.
Probability Theory and Related Fields 168, (2017), 691--719. 

\bibitem[AS19]{AS19} Asselah Amine; Schapira Bruno. Moderate deviations for the range of a transient walk. II. arxiv. 

\bibitem[AS20a]{AS20a} Asselah Amine; Schapira Bruno. On the nature of the Swiss cheese in dimension $3$. Ann. Probab. to appear. 

\bibitem[AS20b]{AS20b} Asselah Amine; Schapira Bruno. Random Walk, Local times, and subsets maximizing capacity, arXiv:2003.03073. 

\bibitem[AS20c]{AS20c} Asselah Amine; Schapira Bruno. Large Deviations for Intersections of Random Walks, arXiv:2005.02735. 

\bibitem[ASS18]{ASS18a} Asselah Amine; Schapira Bruno; Sousi Perla. 
Capacity of the range of random walk. 
Trans. Amer. Math. Soc. 370 (2018), 7627--7645.

\bibitem[ASS19a]{ASS18b} Asselah Amine; Schapira Bruno; Sousi Perla. 
A strong law of large numbers for the capacity of the Wiener sausage in dimension four. 
Probab. Theory Related Fields 173 (2019), 813--858. 

\bibitem[ASS19b]{ASS19} Asselah Amine; Schapira Bruno; Sousi Perla. 
Capacity of the range of random walk on $\Z^4$.
Ann. Probab. 47 (2019), 1447--1497. 


\bibitem[BBH01]{BBH01} van den Berg, Michiel; Bolthausen, Erwin; den Hollander, Frank.  Moderate deviations for the volume of the Wiener sausage. 
The Annals of Mathematics (2) 153 (2001), 355--406.





\bibitem[C17]{Chang} Chang Yinshan. Two observations on the capacity of the range of simple random walk on $\Z^3$ and $\Z^4$. Electron. Commun. Probab. 22, (2017). 



\bibitem[Chen10]{Chen} Chen, Xia. Random walk intersections. Large deviations and related topics. Mathematical Surveys and Monographs, 157. American Mathematical Society, Providence, RI, (2010). x+332 pp.

\bibitem[CL06]{CL} Chung Fan, Lu Linyuan. Concentration inequalities and martingale inequalities: a survey. Internet Math. 3 (2006), 79--127. 




\bibitem[DZ98]{DZ} Dembo A., Zeitouni O. Large deviations techniques and applications. Second edition. Applications of Mathematics (New York), 38. Springer-Verlag, New York, 1998. xvi+396 pp.


\bibitem[HK]{HK} Hamana Yuji, Kesten Harry.
A large-deviation result for the range of random
walk and for the Wiener sausage.
Probab. Theory Relat. Fields 120, 183--08 (2001)

\bibitem[Hut18]{Hut} Hutchcroft Tom. 
Universality of high dimensional spanning forests and sandpiles. 
Preprint 2018, arXiv 1804.04120

\bibitem[JO69]{JO} Jain Naresh C., Orey Steven. On the range of random walk. Israel J. Math. 6, 1968, 373--380 (1969). 

\bibitem[JP71]{JP} Jain Naresh C.; Pruitt William E. The range of transient random walk. J. Analyse Math. 24, (1971), 369--393.

\bibitem[Law91]{Law91} Lawler, Gregory F. Intersections of random walks. Reprint of the $1996$ edition. Modern Birkhauser Classics. Birkhauser/Springer, New York, 2013. iv+223

\bibitem[LawSW18]{LawSW} Lawler Gregory F., Sun Xin, Wu Wei. Loop-erased random walk, uniform spanning forests and bi-Laplacian Gaussian field in the critical dimension, arXiv:1608.02987. 


\bibitem[Na69]{anagaev} Nagaev, A.V.  Integral limit theorems for large deviations when Cramer's condition is not fulfilled (Russian)I,II  Teor. Verojatnost. i Primenen.  14  (1969) 51--64, 203--216.



\bibitem[LG86]{LG86} Le Gall, J.-F. Propri\'et\'es d'intersection des marches al\'eatoires. I. Convergence vers le temps local d'intersection. Comm. Math. Phys. 104 (1986), 471--507. 







\bibitem[Sch19]{Sch19} Schapira, Bruno. Capacity of the range in dimension $5$, arXiv:1904.11183. 

\bibitem[S10]{S10} Sznitman, Alain-Sol. Vacant set of random interlacements and percolation, Ann. of Math. (2) 171, (2010), 2039--2087. 

\end{thebibliography}
\end{document}